\documentclass{article}

\usepackage{amsmath,amssymb}
\usepackage{amsthm}
\usepackage{epsfig,graphics}
\usepackage{booktabs}
\usepackage{comment}
\usepackage{url}

\graphicspath{{./figures/}}

\usepackage{algorithm}
\usepackage{algpseudocode}

\usepackage{array,booktabs}
\usepackage{siunitx}

\usepackage{caption} 
\usepackage{eqparbox}
\newdimen{\algindent}
\algnewcommand{\LeftComment}[1]{\Statex \hspace{\algindent}
\( \triangleright \) #1}

\newtheorem{definition}{Definition}

\renewcommand{\le}{\leqslant}
\renewcommand{\ge}{\geqslant}

\newcommand{\Z}{\ensuremath{\mathbb{Z}}}

\def\x{{\bf x}}
\def\y{{\bf y}}
\def\z{{\bf z}}
\def\s{{\bf s}}
\def\t{{\bf t}}
\def\c{{\bf c}}
\def\j{{\bf j}}
\def\k{{\bf k}}
\newcommand\bxi{\boldsymbol{\xi}}

\def\dt{{\Delta t}}

\def\I{\mathcal{I}}
\def\V{\mathcal{V}}
\def\cd{\mathcal{D}}
\def\cs{\mathcal{S}}

\def\n{{\bf n}}
\def\k{{\bf k}}
\def\ub{{\bf u}}

\def\beq{\begin{equation}}
\def\eeq{\end{equation}}

\def\Z{{\bf Z}}

\def\rprime{\hbox{\hskip.25em\raise.5ex\hbox{$'$}\hskip.15em}}

\def\ddn1{{\frac{\partial}{\partial \nu_{\yb}}}}

\def\ft{{\tilde f}}

\newcommand{\bbR}{{\mathbb R}}

\newcommand{\cn}{\cite}
\newcommand{\ignore}[1]{}

\newtheorem{theorem}{Theorem}[section]

\newtheorem{lemma}[theorem]{Lemma}
\newtheorem{corollary}[theorem]{Corollary}
\newtheorem{remark}[theorem]{Remark}

\newcommand{\be}{\begin{equation}}
\newcommand{\ee}{\end{equation}}
\newcommand{\ba}{\begin{aligned}}
\newcommand{\ea}{\end{aligned}}
\newcommand{\bea}{\begin{eqnarray}}
\newcommand{\eea}{\end{eqnarray}}

\DeclareMathOperator{\sgn}{sgn}
\DeclareMathOperator{\erfc}{erfc}

\newcommand*{\affaddr}[1]{#1} 
\newcommand*{\affmark}[1][*]{\textsuperscript{#1}}

\def\ymsc{Yau Mathematical Sciences Center, \\
Tsinghua University,\\
Haidian District, Beijing, China 100084}

\def\FI{Flatiron Institute, Simons Foundation, \\ New York, New York 10010}
\def\njit{Department of Mathematics Sciences, \\
New Jersey Institute of Technology,\\
Newark, New Jersey 07102}
\def\nyu{Courant Institute, New York University,\\ New York, New York 10012}
\def\umich{Department of Mathematics, University of Michigan, \\ Ann Arbor, Michigan 48109}
\title{Fast integral equation methods for linear and semilinear
heat equations in moving domains}

\author{%
J. Wang\affmark[1], L. Greengard\affmark[2,3], S. Jiang\affmark[4] 
and S. Veerapaneni\affmark[5]\\
$\,$ \\
\affaddr{\affmark[1]\ymsc}\\
Email: {\tt jwang2020@tsinghua.edu.cn} \\ 
$\,$\\
\affaddr{\affmark[2]\FI} \\ 
\affaddr{\affmark[3]\nyu}\\
Email: {\tt greengard@courant.nyu.edu} \\
$\,$ \\
\affaddr{\affmark[4]\njit}\\ 
Email: {\tt shidong.jiang@njit.edu} \\
$\,$\\
\affaddr{\affmark[5]\umich}\\
Email: {\tt shravan@umich.edu}
}

\begin{document}

\maketitle

\begin{abstract}
We present a family of integral equation-based solvers 
for the linear or semilinear heat equation in complicated moving 
(or stationary) geometries. This approach has significant advantages 
over more standard finite element or finite difference methods in terms 
of accuracy, stability and space-time adaptivity. In order to be practical, 
however, a number of technical capabilites are required: fast algorithms 
for the evaluation of heat potentials, high-order accurate quadratures 
for singular and weakly integrals over space-time domains, and robust 
automatic mesh refinement and coarsening capabilities.  We describe 
all of these components and illustrate the performance of 
the approach with numerical examples in two space dimensions.
\end{abstract}

\tableofcontents 
\vspace{3mm}
\section{Introduction \label{sec:intro}}  

A variety of problems in applied physics, chemistry and 
engineering require the solution of the diffusion equation
\begin{equation} \label{heatfree}
\begin{split}
u_{t}(\x,t) &= \Delta u(\x,t)+F(u,\x,t), \\
u(\x,0) &= u_0(\x),
\end{split}
\end{equation}
for $t>0$ and $\x \in \mathbb{R}^d$.
The problem is sometimes posed in free space, but often
it is assumed that $(\x,t)$ lies in 
the interior or exterior of a space-time domain 
\begin{equation}
\Omega_T=\prod_0^T \Omega(t)\subset \mathbb{R}^d\times (0,T], 
\label{domaindef}
\end{equation}
subject to suitable conditions on 
its boundary 
\begin{equation}
\Gamma_T = \prod_0^T \Gamma(t), 
\label{boundarydef}
\end{equation}
where $\Gamma(t)=\partial \Omega(t)$.
As a matter of nomenclature, we refer to the  
equation \eqref{heatfree} as {\em semilinear} and {\em nonautonomous}.
When $F(u,\x,t) = F(u)$, the equation is {\em autonomous}, and when
$F(u,\x,t) = F(\x,t)$, it is {\em linear}.

Classical methods, such as finite difference or finite element methods,
evolve the solution at successive time steps, using some local 
approximation of the partial differential equation (PDE)
\eqref{heatfree} itself.
While more general,
they are known to have a severe constraint on the size of time step when
relying on an explicit marching scheme. Using the forward Euler method,
for example, if we approximate the Laplacian via the second order
central difference operator on a uniform grid with
the same spacing in each coordinate direction, then
the stability condition for solving the pure initial value problem requires
that 
$\Delta t \leq \frac{1}{2d}\Delta x^2$, where $\Delta t$ is the time 
step, $\Delta x$ is the step in the spatial discretization, and $d$
is the ambient dimension.
In general, the stability constraint takes the form
$\dt=\mathcal{O}(h^2)$ where $h$ is the smallest spacing in the
spatial discretization.
This constraint is typically overcome
by using an implicit marching scheme. Even in the linear case, however, this
involves the solution of a system of algebraic equations at
each time step, coupling all the spatial grid points. In the semilinear
case, it requires the solution of a large, nonlinear system of equations. 
Moreover, high order accuracy is difficult to achieve with direct
discretization methods when dealing with 
time-dependent domains.
The available literature on such approaches is vast and we do not seek
to review it here.
In this paper, we focus on the use of parabolic potential theory
\cite{costabel,friedman1964,guenther1988,pogorzelski} and the development 
of the fast algorithms that will permit us to solve fairly
general initial-boundary value problems. 

Let us first consider the linear setting, when the forcing term $F(u,\x,t) = F(\x,t)$
in \eqref{heatfree} is known, in the absence of physical boundaries.
For the sake of simplicity, we assume that the forcing term at time
$t$ is supported in the bounded domain $\Omega(t)$ and that the initial
data $u_0(\x)$ is supported in $\Omega(0)$. This is a  well-posed problem
in free space under mild conditions on the behavior of $u$ at infinity
\cite[p. 25]{friedman1964}. The solution to \eqref{heatfree} can be expressed
at any later time $t$ in closed form as
\begin{equation}
u^{(V)}(\x,t) = \I[u_0](\x,t)+\V[F](\x,)
\label{freesol}
\end{equation}
with 
\begin{equation}
\I[u_0](\x,t) = \int_{\Omega(0)} G(\x-\y,t) u_0(\y) \, d\y, 
\label{initpot}
\end{equation}
and 
\begin{equation}
\V[F](\x,t) = \int_0^t\int_{\Omega(\tau)} G(\x-\y,t-\tau) F(\y,\tau) \, d\y d\tau.
\label{volpot}
\end{equation}
Here,
\begin{equation}
G(\x,t) = \frac{e^{-\|\x\|^2/4t}}{(4 \pi t)^{d/2}}
\label{heatker}
\end{equation}
is the free-space Green's function for the heat equation 
(the {\em heat kernel}) in $d$ dimensions,
The functions $\I[u_0]$ and $\V[F]$ are referred to as {\em initial} and
{\em volume} (heat) potentials, respectively.

\begin{remark}
The spatial domains of integration in $\I[u](\x,t)$ and 
$\V[F](\x,t)$ are implicitly assumed to be the support of the indicated
functions. Above, these are finite domains defined by $\Omega(0)$ and
$\Omega_t$, respectively. 
\end{remark}

\begin{remark}
In the remainder of this paper, we will assume $d=2$. 
\end{remark}

A compelling feature of the solution 
$u^{(V)}$ is that it is exact, explicit, and does not involve the inversion
of any matrices. It requires only the 
{\em evaluation} of the initial
and volume potentials. Thus, stability never arises as an issue, and the error
is simply the quadrature error in approximating the integral operators
in the relevant potentials themselves.
This approach is typically ignored for two simple reasons.
First, good quadrature rules are needed and second, in the absence
of fast algorithms, the computational cost is excessive. More
precisely, assuming there are $N_T$ time steps,
$N$ points in the discretization of $F$ and 
$N$ target points at each time step, naive summation
applied to \eqref{volpot} requires $O(N^2 N_T^2)$ work.

For a time step $\Delta t$, however, 
using the definition of the 
heat kernel, it is straightforward to see that
\begin{equation}
\begin{aligned}
u^{(V)}(\x,t+\delta) = 
&\int_{\mathbb{R}^d} G(\x-\y,\delta) \, 
u^{(V)}(\y,t) \, d\y + \\
&\int_{t}^{t+\delta}
\int_{\Omega(\tau)} G(\x-\y,t+\delta-\tau) \, F(\y,\tau) \, d\y d\tau.
\end{aligned}
\label{vmarch}
\end{equation}
That is to say, the heat equation in free space can be solved by 
sequential convolution with the heat kernel at every time step,
followed by the addition of a {\em local} contribution from 
the forcing term $F(\x,t)$ over the last time step alone.
This is simply the semigroup property satisfied by the heat equation
itself, and can be viewed as a marching scheme for the PDE expressed in 
integral form. Using this approach, the apparent
history-dependence of evaluating heat potentials has vanished and the 
net cost has dropped from
$O(N^2 N_T^2)$ to $O(N^2 N_T)$ work, assuming some suitable quadrature
rule has been developed for the integrals in \eqref{vmarch}.
There is, however, a small catch. Note that the first spatial integral
in \eqref{vmarch} is of the form
\[
\I[u^{(V)}(\x,t)](\x,t+\delta) = 
\int_{\mathbb{R}^d} G(\x-\y,\delta) 
u^{(V)}(\y,t) \, d\y, 
\]
with the domain of integration $\mathbb{R}^d$,
whereas the domain of integration
in the exact solution \eqref{freesol} is the bounded region $\Omega_t$ - 
the support of the data itself. 
Thus, in order to make use of the marching scheme
\eqref{vmarch}, we will have to cope with the spreading of the solution in
free space over time. This turns out to be surprisingly easy to do 
(see section \ref{sec:review}).

In many applications, physical boundaries are present on which 
additional conditions are imposed.
We will limit our attention to three standard boundary value problems.
For the Dirichlet problem, the value
of the solution on the boundary is given:
\begin{equation}
u(\x,t) = 
f(\x,t) \qquad (\x,t) \in \Gamma_T.
\label{dirichlet00}
\end{equation}
For the Neumann problem, the value of the normal derivative 
of the solution at the boundary is given:
\begin{equation}
 \frac{\partial u}{\partial \nu_{\x}}(\x,t) = 
g(\x,t) \quad (\x,t) \in \Gamma_T,
\label{neumann00}
\end{equation}
where $\nu_{\x}$ denotes the 
normal vector at the boundary point $\x$.
For the Robin problem, a specific linear combination of the 
solution and its normal derivative is given on the boundary:
\begin{equation}
  \frac{\partial}{\partial \nu_{\x}}u(\x,t)+\alpha(\x,t)u(\x,t)=h(\x,t)
  \qquad (\x,t)\in\Gamma_T ,
  \label{robin00}
\end{equation}
with $\alpha(\x,t)\geq 0$.

In standard finite difference and finite element methods, 
the boundary condition must be coupled with
the interior marching scheme, which can impose additional
stability constraints on the time step~\cite{gustafsson,gustafsson1972mcom,osher1969mcom,strikwerda,trefethen1983jcp}. In the framework of potential theory,
on the other hand, we may write the solution to the
full (linear) initial-boundary value problem as
\be
u=u^{(V)}+u^{(B)},
\ee
where $u^{(V)}$ is given by \eqref{freesol} and
$u^{(B)}$ is the solution
to a {\em homogeneous} boundary value problem with zero initial data and 
no forcing term. The boundary condition for $u^{(B)}$ is simply the 
original boundary condition for $u(\x,t)$, 
from which is subtracted the contribution 
of $u^{(V)}$. We will see below that we can represent 
$u^{(B)}$ in terms of layer heat
potentials defined in terms of unknown densities which are restricted to 
the space-time boundary $\Gamma_T$. 
These densities are determined by the boundary data 
through the solution of Volterra integral
equations of the second kind.
This formulation has some remarkable benefits.
First, explicit marching schemes are not subject to a mesh-dependent
stability constraint. In fact, for the Dirichet and Neumann problems, 
they can be unconditionally stable, as discussed in section \ref{sec:stabmarch}.
Second, integral equations are naturally compatible with moving boundaries, 
and apply equally well to interior or exterior problems. In the latter
case, they do not 
require the artificial truncation of the computational domain or the associated
difficulties in constructing suitable ``outgoing" boundary conditions. 
As for the volume potential discussed above, however, there are
two principal drawbacks to making such computations practical: the exhorbitant
computational cost and the need to integrate singular or weakly singular 
kernels over complicated surfaces in space-time. 
A naive algorithm would again require an amount of work
of the order $\mathcal{O}(N N_S N_T^2)$ where $N$ is the number of 
target points at each time step, $N_S$ is the number of discretization
points on the boundary $\Gamma(t)$ and $N_T$ is the number of time steps.

\subsection{Semilinear problems}

For the full, non-autonomous semilinear problem \eqref{heatfree} in free space,
we may still use the representation \eqref{freesol}, but now
\begin{equation}
\V[F](\x,t) = \int_0^t\int_{\Omega(\tau)} G(\x-\y,t-\tau) F(u,\y,\tau) \, d\y d\tau.
\label{volpotnonlin}
\end{equation}

This still permits the analog of the marching scheme \eqref{vmarch}, 
except that
\begin{equation}
\begin{aligned}
u(\x,t+\delta) = 
&\int_{\mathbb{R}^d} G(\x-\y,\delta)  
u(\y,t) \, d\y\ + \\
&\int_{t}^{t+\delta}
\int_{\Omega(\tau)} G(\x-\y,t+\delta-\tau) F(u,\y,\tau) \, d\y d\tau.
\end{aligned}
\label{vmarchnonlin}
\end{equation}
is now a nonlinear integral equation for $u(\x,t)$ rather than an 
explicit, exact solution. We discuss the solution to such problems in 
section \ref{sec:alg}.

\subsection{Periodic boundary conditions} \label{periodicbc}

In many applications, the solution of the heat equation is required
with periodic boundary conditions on, say, the unit box
$B = [-\frac{1}{2},\frac{1}{2}]\times [-\frac{1}{2},\frac{1}{2}]$.
That is, $u(\x,t)$ must satisfy
\[
\begin{aligned}
u(-\tfrac{1}{2},y) &= u(\tfrac{1}{2},y), \quad
u_x(-\tfrac{1}{2},y) = 
u_x(\tfrac{1}{2},y),\ {\rm for}\ 
-\tfrac{1}{2} \leq y \leq \tfrac{1}{2} \\
u(x,-\tfrac{1}{2}) &= u(x,\tfrac{1}{2}), \quad
u_y(x,-\tfrac{1}{2}) = 
u_y(x,\tfrac{1}{2}),\ {\rm for}\ 
-\tfrac{1}{2} \leq x \leq \tfrac{1}{2}.
\end{aligned}
\]
In the absence of physical boundaries, 
the solution to the linear heat equation with forcing
is again of the form \eqref{freesol}, but with $G(\x,t)$ defined to be the 
periodic Green's function for the heat equation in two dimensions:
\begin{equation}
G_P(\x,t) =
\sum_{j=-\infty}^{\infty}
\sum_{k=-\infty}^{\infty}
 \frac{e^{-\|\x + (j,k) \|^2/4t}}{(4 \pi t)},
\label{heatker2dper}
\end{equation}
using the method of images.
The marching scheme in \eqref{vmarch} still applies, except the
domain of integration is 
the box $B$ for both the initial potential and the 
local increment, rather than ${\mathbb{R}^2}$ and $\Omega(t)$.
Moreover, the marching scheme
\eqref{vmarchnonlin} applies in the semilinear case here as well (with 
$G_P(\x,t)$ replacing the free space Green's function $G(\x,t)$).

We will also consider boundary value problems where a collection of
inclusions define a region $\Omega_T \subset B$, and the domain of 
interest is the complement $B \setminus \Omega_T$, with 
periodic conditions on $B$ and Dirichlet, Neumann or Robin conditions
on $\Gamma_T$ (see Section \ref{sec:examples}).
\begin{figure}
\centering
\includegraphics[width=.6\textwidth]{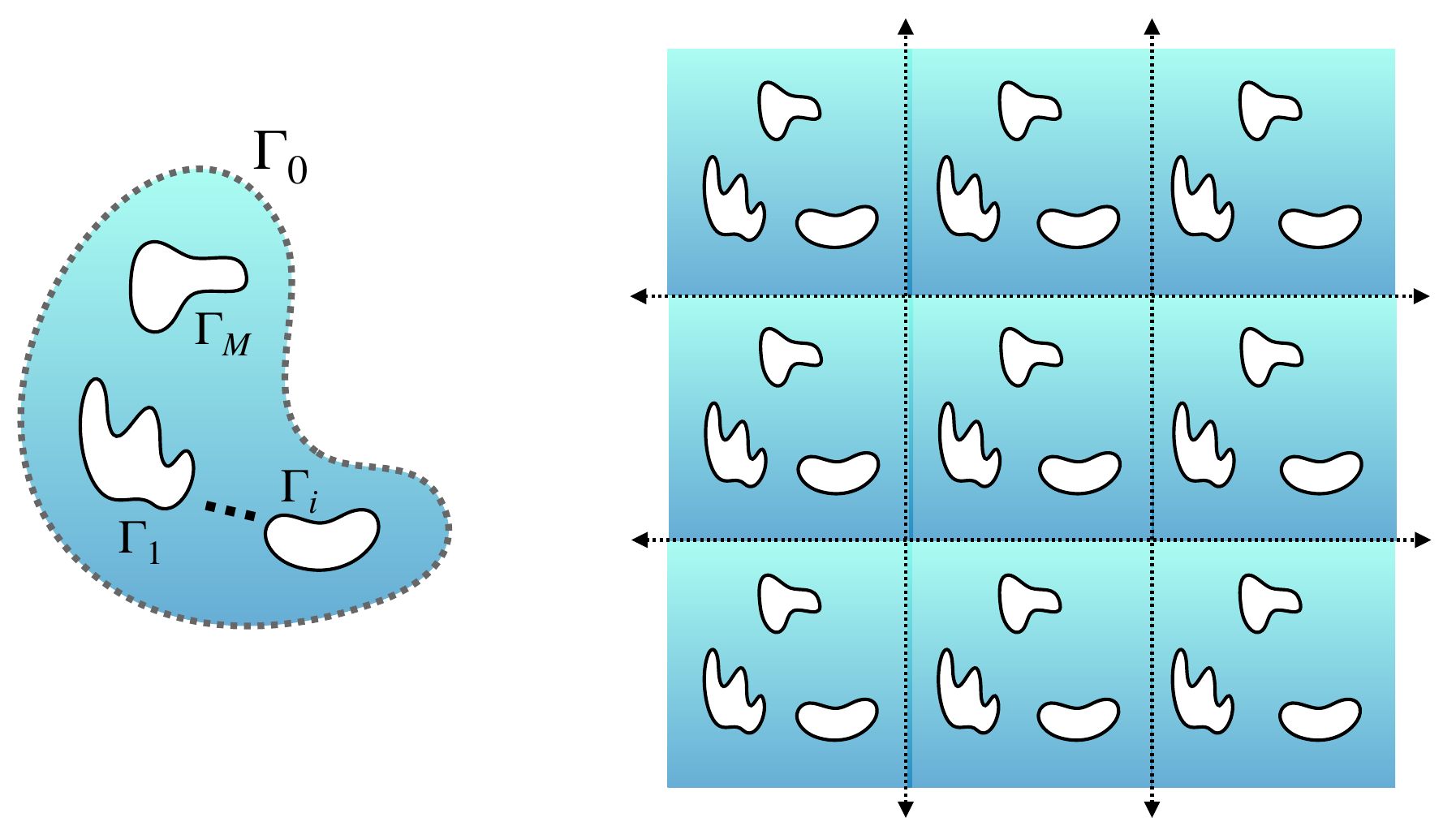}
\caption{
Typical domains of interest:
on the left is a multiply connected region defined by the closed 
curves $\Gamma_1(t),\dots,\Gamma_M(t)$. 
Interior problems include an outer boundary $\Gamma_0(t)$
and the domain defined at time $t$
by $\Gamma(t) = \cup _{i = 0}^M \Gamma_i(t)$.
We will also consider 
boundary value problems in the exterior
domain bounded at time $t$
by $\Gamma(t) = \cup _{i = 1}^M \Gamma_i(t)$.
Periodic boundary value problems are illustrated
on the right.
We show the unit box $B$ and its eight nearest neighbors in an infinite
tiling of the plane. When inclusions are present, the domain of interest
is $B \setminus \Omega_T$, where $\Omega_T$ is the region bounded by the 
indicated curves. We will also consider periodic problems 
in the absence of inclusions, defined by initial data and a linear or 
semilinear forcing term.}
\label{domains}
\end{figure}

\subsection{Synopsis of the paper}

The paper is organized as follows. In section~\ref{sec:prelim}, we review
the relevant parts of parabolic 
potential theory and show how to reformulate the various boundary value problems
as Volterra integral equations of the second kind. 
In section~\ref{sec:review} and \ref{sec:bootstrap}, 
we discuss the state of the art 
in fast algorithms for the evaluation of 
heat potentials, focusing on some recently developed quadrature methods and
adaptive versions of the fast Gauss transform (FGT) 
\cite{brattkus1992siap,greengard1991fgt,greengard1990cpam,greengard1998nfgt,greengard2000acha,Lee2006,li2009sisc,sampath2010pfgt,spivak2010sisc,strain1994sisc,tausch2009sisc,veerapaneni2008jcp,wang2018sisc}.
The full numerical scheme for boundary value problems
is summarized in section~\ref{sec:alg}, and extensive
numerical examples are provided in section~\ref{sec:examples} in both 
the linear and semilinear settings.
We conclude in section~\ref{sec:conclusion} with a discussion of open problems.

We hope to make clear some of the distinct advantages to be gained from
using the integral equation framework. These include the following:
\begin{enumerate}
\item 
Explicit marching schemes, as mentioned above, can be {\em unconditionally stable}
(see section \ref{sec:stabmarch}).
\item
For semilinear problems, fully implicit methods 
require only the solution of scalar nonlinear equations
\cite{epperson2}, unlike PDE-based methods 
(see section \ref{sec:alg}).
\item
Fast algorithms are available with which all of the methods described
are asymptotically optimal and able to achieve arbitrary order accuracy
(see section \ref{sec:review}).
\end{enumerate}
\subsection{Notation}
In the table below, we list some symbols that will be used frequently.

\begin{table}[!htb]
\centering

\begin{tabular}{c|l}
\hline
\textbf{Symbol} & \textbf{Definition} \\ [1 mm]
\hline
$\Omega_T$ & Space-time domain (eq. \ref{domaindef}) \\
$\Gamma_T$ & Space-time boundary (eq. \ref{boundarydef}) \\
$\I$ & Initial potential (eq. \ref{initpot}) \\
$\V$ & Volume potential (eq. \ref{volpot}) \\
$\cs$ & Single-layer potential (eq. \ref{eq:slp}) \\
$\cd$ & Double-layer potential (eq. \ref{eq:dlp}) \\
$\cs_H, \cd_H, \V_H$ & the {\em history} parts of $\cs$, $\cd$ and $\V$ 
(eq. \ref{eq:hist}) \\
$\cs_{NH}, \cd_{NH}$ & the {\em near history} parts of $\cs$ and $\cd$ 
(eq. \ref{bootstrap}) \\
$\cs_{FH}, \cd_{FH}$ & the {\em far history} parts of $\cs$ and $\cd$ 
(eq. \ref{smooth}) \\
$\cs_L, \cd_L, \V_L$ & the {\em local} parts of $\cs$, $\cd$ and $\V$ 
(eq. \ref{eq:loc}) \\
$\cs_\epsilon, \cd_\epsilon$ & the {\em asymptotic} parts of 
$\cs$ and $\cd$ 
(eq. \ref{eq:slocasym}) \\
\hline
\end{tabular}
\caption{Index of frequently used symbols and operators.}
\label{tbl:notation}
\end{table}

\section{Mathematical Preliminaries \label{sec:prelim}}

In this section, we review classical potential theory
for the heat equation and refer the reader to the texts
\cn{friedman1964}, \cn{guenther1988}, and \cn{pogorzelski} for more
thorough treatments.
While the theory is similar to that for the elliptic case (governed
by the Laplace or Helmholtz equations), there are some important differences.
Existence and uniqueness results are easier to obtain
because the integral equations that arise are second kind Volterra 
equations rather than Fredholm equations, avoiding the possibility of
spurious resonances. 

As noted in the introduction, it is well-known 
that the equations \eqref{heatfree} 
are well-posed in free space, with the solution given by \eqref{freesol}.
The solution is unique assuming
$u_0(\x)$ and $F(\x,t)$ are continuous and bounded by some constant 
$L$ on $\bbR^d$ \cite{friedman1964,pogorzelski}.
In that case, the solution $u^{(V)}(\x,t)$ to \eqref{heatfree}
is continuous and satisfies $|u^{(V)}(\x,t)|\leq L$ for $\x\in\bbR^d$ and $t>0$.
\subsection{The classical initial-boundary value problems}

We turn now to a consideration of interior boundary value problems 
in domains $\Omega_T$ or their exterior counterparts, defined in 
$\Omega_T^c=\prod_{t=0}^T \Omega^c(t)$, 
where $\Omega^c(t)=\bbR^d\backslash \Omega(t)$, as depicted in 
Fig. \ref{domains}.
When the spatial domain is stationary, we will denote the space-time
domain on which the heat equation is to be solved by
the interior region
$\Omega_T=\Omega\times [0,T]$, with boundary $\Gamma_T=\Gamma\times [0,T]$
and the corresponding exterior domain as 
$\Omega_T^c$. 
As indicated earlier, we 
denote the unit outward normal to $\Gamma(t)$ at each point $\x \in \Gamma(t)$ 
by $\nu_{\x}$ and the positively-oriented unit tangent vector at $\x$
by $\tau_{\x}$. 
For all boundary value problems of interest, we assume that
\begin{equation}
\label{heatfreebd}
\begin{aligned}
u_t(\x,t) & =\Delta u(\x,t)+F(\x,t), \qquad (\x,t)\in\Omega_T,  \\
u(\x,0) &= u_0(\x), \qquad \x\in \overline{\Omega(0)}.
\end{aligned}
\end{equation}

For the sake of simplicity, we always assume that the initial and boundary data
are compatible.
For the {\em Dirichlet} problem, where we impose 
\eqref{dirichlet00},
this requires that
\[
f(\x,0)=u_0(\x), \qquad \x\in\Gamma(0). 
\]
For the {\em Neumann} problem, where we impose
\eqref{neumann00},
this requires that
\[
g(\x,0)=\frac{\partial}{\partial \nu_{\x}}u_0(\x), \qquad \x\in\Gamma(0) .
\]
Finally, for the {\em Robin} problem, where we impose
\eqref{robin00},
this requires that 
\[
h(\x,0)=\frac{\partial}{\partial 
\nu_{\x}}u_0(\x)+\alpha(\x)u_0(\x), \qquad \x\in\Gamma(0) .
\]

\subsection{Layer Heat Potentials\label{sec:layer}}

As in the elliptic case, homogeneous 
partial differential equations of parabolic type
can be reduced to 
boundary integral equations 
when the material coefficients (here the conductivities) are constant.
This transformation makes essential use of the governing Green's function
and a source distribution restricted to the space-time boundary $\Gamma_T$,
called a surface density or layer potential density.
Formally speaking, these are distributions when viewed as functions in the 
ambient space (like a surface charge or surface current density in 
electromagnetic theory) but simply defined as 
functions on the surface using a suitable
parametrization.

\begin{definition}
In the remainder of this paper, we assume the domain $\Omega_T$
is {\em sufficiently smooth}, meaning that
the domain boundary
$\Gamma(t)$ is $C^2$ with continuous curvature and
that, in the moving geometry case, it evolves with a continuous velocity.
\end{definition}

\begin{definition}
Suppose that $\Omega_T$ is a sufficiently smooth, 
bounded space-time domain and that the single and double layer
densities $\sigma(\x,t)$ and $\mu(\x,t)$ are continuous 
for $(\x,t)\in \Gamma_T$. Then,
the {\em single layer potential} 
$\cs[\sigma]$ is given by
\begin{equation}\label{eq:slp}
\cs[\sigma](\x,t)=
\int_0^t\int_{\Gamma(\tau)} G(\x-\y,t-\tau) \sigma(\y,\tau) ds_{\y}d\tau,
\end{equation}
and the double layer potential $\cd[\mu]$ is given by
\begin{equation}\label{eq:dlp}
\cd[\mu](\x,t)=\int_0^t\int_{\Gamma(\tau)} \frac{\partial}{\partial \nu_{\y}} G(\x-\y,t-\tau) \mu(\y,\tau) ds_{\y}d\tau\, .
\end{equation}
\end{definition}

The layer heat potentials defined in (\ref{eq:slp}) and (\ref{eq:dlp}) 
satisfy well-known jump conditions as $\x$ approaches a point $\x_0$ on 
the boundary $\Gamma(t)$ from either the interior or exterior. 
Their derivations are standard \cite{friedman1964,guenther1988,pogorzelski}
and easily deduced from their better known elliptic analogs. 

\begin{theorem}\label{thm:slpjump}
The single layer potential $\cs[\sigma](\x,t)$
is continuous for all $\x\in\bbR^2$ and $t\geq 0$ and 
satisfies the relations
\begin{gather}
\frac{\partial}{\partial t} \cs[\sigma](\x,t)=\Delta \cs[\sigma](\x,t),\;\;\;x\notin \Gamma(t),\;\;\;t>0, \\
\cs[\sigma](\x,0)=0,\;\;\;x\notin \Gamma(t), \\
\lim_{\substack{\x\rightarrow\x_0\in \Gamma(t) \\ \x\in \Omega(t)}} \frac{\partial}{\partial \nu_{\x_0}}\cs[\sigma](\x,t)=\frac{1}{2}\sigma(\x_0,t)+\frac{\partial}{\partial \nu_{\x_0}}\cs[\sigma](\x_0,t), \label{eq:slpjumpint}\\
\lim_{\substack{\x\rightarrow\x_0\in \Gamma(t) \\ \x\in \Omega^c(t)}} \frac{\partial}{\partial \nu_{\x_0}}\cs[\sigma](\x,t)=-\frac{1}{2}\sigma(\x_0,t)+\frac{\partial}{\partial \nu_{\x_0}}\cs[\sigma](\x_0,t) \label{eq:slpjumpext} \, ,
\end{gather}
where the expression
$\frac{\partial}{\partial \nu_{\x_0}}\cs[\sigma](\x_0,t)$  
is interpreted in the principal value sense.
\end{theorem}

\begin{theorem}\label{thm:dlpjump}
The double layer potential $\cd[\mu](\x,t)$ 
is defined for all $\x\in\bbR^d$ and $t\geq 0$ and 
satisfies the relations
\begin{gather}
\frac{\partial}{\partial t} \cd[\mu](\x,t)=\Delta \cd[\mu](\x,t),\;\;\;x\notin \Gamma(t),\;\;\;t>0, \\
\cd[\mu](\x,0)=0,\;\;\;x\notin \Gamma(t), \\
\lim_{\substack{\x\rightarrow\x_0\in \Gamma(t) \\ \x\in \Omega(t)}} \cd[\mu](\x,t)=-\frac{1}{2}\mu(\x_0,t)+\cd[\mu](\x_0,t), \label{eq:dlpjumpint}\\
\lim_{\substack{\x\rightarrow\x_0\in \Gamma(t) \\ \x\in \Omega^c(t)}} \cd[\mu](\x,t)=\frac{1}{2}\mu(\x_0,t)+\cd[\mu](\x_0,t), \label{eq:dlpjumpext} 
\end{gather}
where the expression
$\cd[\mu](\x_0,t)$  
is interpreted in the principal value sense.
\end{theorem}

\subsection{Integral representations for 
boundary value problems}
\label{sec:three}

Suppose now that we seek to solve a linear initial-boundary value 
problem and that
$u^{(V)}$ satisfies \eqref{heatfreebd}. If we let 
\be
u^{(B)}(\x,t) = u(\x,t)-u^{(V)}(\x,t),
\label{udecomp}
\ee
then $u^{(B)}$ satisfies 
the {\em homogeneous} heat equation with zero initial data with
boundary data modified by the contribution of $u^{(V)}$. This solution
can be represented by a suitable layer potential
with the unknown density restricted to the space-time boundary itself.
There are numerous possible representations for $u^{(B)}(\x,t)$, 
each leading to a different boundary integral equation. 
So-called direct methods are based 
on the application of Green's identity, from which the unknown is either
$u^{(B)}$  or its normal derivative 
\cite{arnold1989jcm,costabel,friedman1964}.
Here, we focus on the ``indirect'' approach, where the density does not 
necessarily have 
a simple physical interpretation \cite{pogorzelski}.
This approach is more flexible, and leads to 
well-conditioned integral equations for all of the boundary value problems
of interest.
\subsubsection{The Dirichlet Problem}
For the interior Dirichlet problem, it is standard to represent
$u^{(B)}(\x,t)$ as a double layer potential, with the full
solution written as
\begin{equation}\label{eq:ansatzdir}
u(\x,t)=u^{(V)}(\x,t)+u^{(B)}(\x,t)=u^{(V)}(\x,t)+\cd[\mu](\x,t).
\end{equation}
Letting $\x\in\Omega(t)$ approach $\x_0\in\Gamma(t)$ from the interior,
and applying the jump relation (\ref{eq:dlpjumpint}), we obtain
an integral equation for the density function $\mu$:
\begin{equation}\label{eq:intdir}
\frac{1}{2}\mu(\x_0,t)-\cd[\mu](\x_0,t)=u^{(V)}(\x_0,t)-f(\x_0,t) \, .
\end{equation}

The equation \eqref{eq:intdir} is a
Volterra equations of the second kind, and uniqueness is established
by a fixed point argument, analogous to Picard iteration for ordinary
differential equations.
Using the initial guess 
$\mu_0(\x_0,t) = 0$, consider the
sequence of density functions given by the recursion:
\begin{equation}\label{eq:fixedptiter}
\mu_{n+1}(\x_0,t)=2 (u^{(V)}(\x_0,t) - f(\x_0,t))+2\cd[\mu_n](\x_0,t).
\end{equation}
It is straightforward to prove that the sequence of functions ${\mu_n}$ 
converges uniformly on the boundary $\Gamma_T$, and that the limit 
$\mu^{*}$ is a solution to equation (\ref{eq:intdir}). 
We refer the reader to \cn{pogorzelski} for details.
Once the solution is obtained, substitution into
(\ref{eq:ansatzdir}) yields the desired solution $u(\x,t)$ at an arbitrary
location.

For the {\em exterior} Dirichlet problem, let us assume that
$u(\x,t)$ is bounded at infinity and that the initial data $u_0(\x)$ 
and forcing term $F(\x,t)$ have compact support. 
Using the same ansatz \eqref{eq:ansatzdir}, and letting $\x$ 
approach $\x_0\in\Gamma(t)$ from the exterior, the jump relation
(\ref{eq:dlpjumpext}) yields the second kind Volterra equation:
\begin{equation}\label{eq:extdir}
\frac{1}{2}\mu(\x_0,t)+D[\mu](\x_0,t)=f(\x_0,t)-u^{(V)}(\x_0,t).
\end{equation}
Existence and uniqueness can be proven by the same fixed point argument.
\subsubsection{The Neumann Problem}

For the Neumann problem, we 
represent $u^{(B)}(\x,t)$ as a single
layer potential:
\begin{equation}\label{eq:ansatzneu}
u(\x,t)=u^{(V)}(\x,t)+u^{(B)}(\x,t)=u^{(V)}(\x,t)+\cs[\sigma](\x,t).
\end{equation}
Letting $\x\in\Omega(t)$ approach a point $\x_0\in\Gamma(t)$ from the 
interior and applying the jump relation (\ref{eq:slpjumpint}), we obtain
the integral equation
\begin{equation}\label{eq:intneu}
  \frac{1}{2}\sigma(\x_0,t)+\frac{\partial}{\partial \nu_{\x_0}}\cs[\sigma](\x_0,t)
  =g(\x_0,t)-\frac{\partial}{\partial \x_0}u^{(V)}(\x_0,t).
\end{equation}
Existence and uniqueness for (\ref{eq:intneu}) are established
as above.

For the exterior problem, we assume again that $u(\x,t)$ is bounded 
at infinity and that the initial
data $u_0(\x)$ and the forcing term $F(\x,t)$ have compact support. 
Using the representation (\ref{eq:ansatzneu}), the
jump relation (\ref{eq:slpjumpext}) leads to the integral equation
\begin{equation}\label{eq:extneu}
  -\frac{1}{2}\sigma(\x_0,t)+\frac{\partial}{\partial \nu_{\x_0}}\cs[\sigma](\x_0,t)
  =g(\x_0,t)-\frac{\partial}{\partial \x_0}u^{(V)}(\x_0,t).
\end{equation}
\subsubsection{The Robin Problem}

For the Robin problem, it is convenient to use the same 
representation as for the Neumann problem.  Thus, we let
\begin{equation}\label{eq:ansatzrob}
u(\x,t)=u^{(V)}(\x,t)+u^{(B)}(\x,t)=u^{(V)}(\x,t)+\cs[\sigma](\x,t).
\end{equation}
For the interior problem, letting $\x$ approach $\x_0\in\Gamma(t)$, and using
the jump relation (\ref{eq:slpjumpint}), we obtain an
integral equation for the unknown density $\sigma$:
\begin{equation}\label{eq:introb}
\begin{split}
\frac{1}{2}\sigma(\x_0,t)&+\left(\frac{\partial}{\partial \nu_{\x_0}}+
\alpha(\x_0) \right)\cs[\sigma](\x_0,t) \\
&=h(\x_0,t)-\left(\frac{\partial}{\partial \x_0}+\alpha(\x_0)\right)
u^{(V)}(\x_0,t).
\end{split}
\end{equation}

For the exterior problem, the jump relation (\ref{eq:slpjumpext}) 
leads to the integral equation
\begin{equation}\label{eq:extrob}
\begin{split}
-\frac{1}{2}\sigma(\x_0,t)&+ \left(\frac{\partial}{\partial \nu_{\x_0}}+\alpha(\x_0) \right)\cs[\sigma](\x_0,t) \\
&=h(\x_0,t)-\left(\frac{\partial}{\partial \x_0}+\alpha(\x_0)\right)
u^{(V)}(\x_0,t) \, .
\end{split}
\end{equation}

It remains to develop efficient solvers for each of the integral equations 
above, which 
will require fast algorithms, high order accurate quadratures, and stable
marching schemes. The latter is important, just as for ordinary differential 
equations, because fixed point iteration is not a practical approach -
both because it is fully history dependent and because it 
converges extremely slowly.
\subsection{Decomposition of heat potentials \label{sec:decomp}}

A standard tool in analyzing heat potentials is to introduce a
cutoff parameter $\delta$
that divides the domain of integration in time into two parts:
a ``history" part, describing the influence at time $t$ of 
the layer or volumetric source densities during the interval $[0,t-\delta]$
and a ``local" part, describing the influence at time $t$ of 
the source potential densities during the most recent interval $[t-\delta,t]$.
For this, we let
\begin{equation}
\label{eq:decomp}
\begin{split}
  \cs[\sigma](\x,t)
  &=
\cs_L[\sigma](\x,t) + \cs_H[\sigma](\x,t), \\
\cd[\mu](\x,t) &=
\cd_L[\mu](\x,t) + \cd_H[\mu](\x,t), \\
\V[F](\x,t) &=
\V_L[F](\x,t) + \V_H[F](\x,t), 
\end{split}
\end{equation}
where
\begin{equation}
\label{eq:hist}
\begin{split}
\cs_H[\sigma](\x,t) &=
\int _{0}^{t-\delta} \int _{\Gamma(\tau)}
G(\x-\y,t-\tau) \;\sigma(\y,\tau) \; ds_{\y} \;d\tau \, , \\
\cd_H[\mu](\x,t) &=
\int_0^{t-\delta} \int_{\Gamma(\tau)} 
\frac{\partial}{\partial \nu_{\y}} 
G(\x-\y,t-\tau) \mu(\y,\tau) ds_{\y}d\tau\, , \\
\V_H[F](\x,t) &=
\int_0^{t-\delta} \int_{\Omega(\tau)} 
G(\x-\y,t-\tau) \: F(\y,\tau) d{\y}d\tau\, ,
\end{split}
\end{equation}
and
\begin{equation}
\label{eq:loc}
\begin{split}
\cs_L[\sigma](\x,t) &=
\int _{t-\delta}^t \int _{\Gamma(\tau)} G(\x-\y,t-\tau) \;
\sigma(\y,\tau) \; ds_{\y} \;d\tau \, , \\
\cd_L[\mu](\x,t) &=
\int_{t-\delta}^t \int_{\Gamma(\tau)} 
\frac{\partial}{\partial \nu_{\y}} 
G(\x-\y,t-\tau) \mu(\y,\tau) ds_{\y}d\tau\, , \\
\V_L[F](\x,t) &=
\int_{t-\delta}^t \int_{\Omega(\tau)} 
G(\x-\y,t-\tau) \; F(\y,\tau) d{\y}d\tau\, .
\end{split}
\end{equation}

Note that
the local part is short-ranged with respect to both space and time, 
but involves a singular kernel. 
The history part, on the other hand, is smooth but history-dependent.
Over the past several decades, a variety of algorithms have been introduced
to address both of these issues, which we present in the next section.
(See \cn{huang2006jcp} and \cn{veerapaneni2007sisc}
for a discussion of fast, integral equation methods in the one-dimensional
setting.)

\section{Fast algorithms for the evaluation of heat potentials} \label{sec:review} 

In this section, we review some of the algorithms that are
presently available for the rapid evaluation of initial, volume and layer heat 
potentials. The simplest case is that of an initial potential 
\eqref{initpot}, which we will refer to as the 
{\em continuous Gauss transform}. Its discrete analog 
takes the form 
\begin{equation}
V(\x_j) = 
\sum_{k=1}^M
G(\x_j - \y_k, \delta) \,  U_k
= \sum_{k=1}^M
\frac{e^{-\|\x_j - \y_k\|^2/4\delta}}{4 \pi \delta}
 U_k,
\label{dgt}
\end{equation}
for $j = 1,\dots,N$.
Algorithms for the efficient computation of 
such {\em discrete Gauss transforms} are generally referred 
to as fast Gauss transforms (FGTs)
\cite{greengard1991fgt,greengard1998nfgt,sampath2010pfgt,spivak2010sisc,tausch2009sisc}.

Volume and layer potentials are quite different in character, as they 
involve integrals over a space-time volume or boundary, respectively. 
Fast algorithms
for these potentials are typically based either on recursion (the Duhamel
principle) as suggested by eq. \eqref{vmarch}, or on hierarchical
compression over time based on the smoothing properties of 
the heat kernel. 

\subsection{The discrete and continuous FGT} \label{sec:dgt}

While the discrete Gauss transform requires $O(MN)$ work by 
direct summation, the FGT permits the evaluation of sums of the form
\eqref{dgt}
using only $O(M+N)$ work, independent of $\delta$.
Continuous versions of the FGT have been developed assuming the
function $u_0(\y)$ in \eqref{initpot} is
given on an adaptive unstructured triangulation
\cite{strain1994sisc}, or when it is given on a 
high-order, adaptive quad-tree based discretization
\cite{veerapaneni2008jcp,wang2018sisc}.

Following the discussion of 
\cite{wang2018sisc}, we briefly describe a hierarchical
version of the FGT that is largely insensitive to $\delta$
and takes as input a user-provided adaptive
quad-tree with either discrete or continuous volume sources,
as illustrated in Fig. \ref{surfadap}.

\begin{figure}[htbp]
\centering
\includegraphics[width=.8\textwidth]{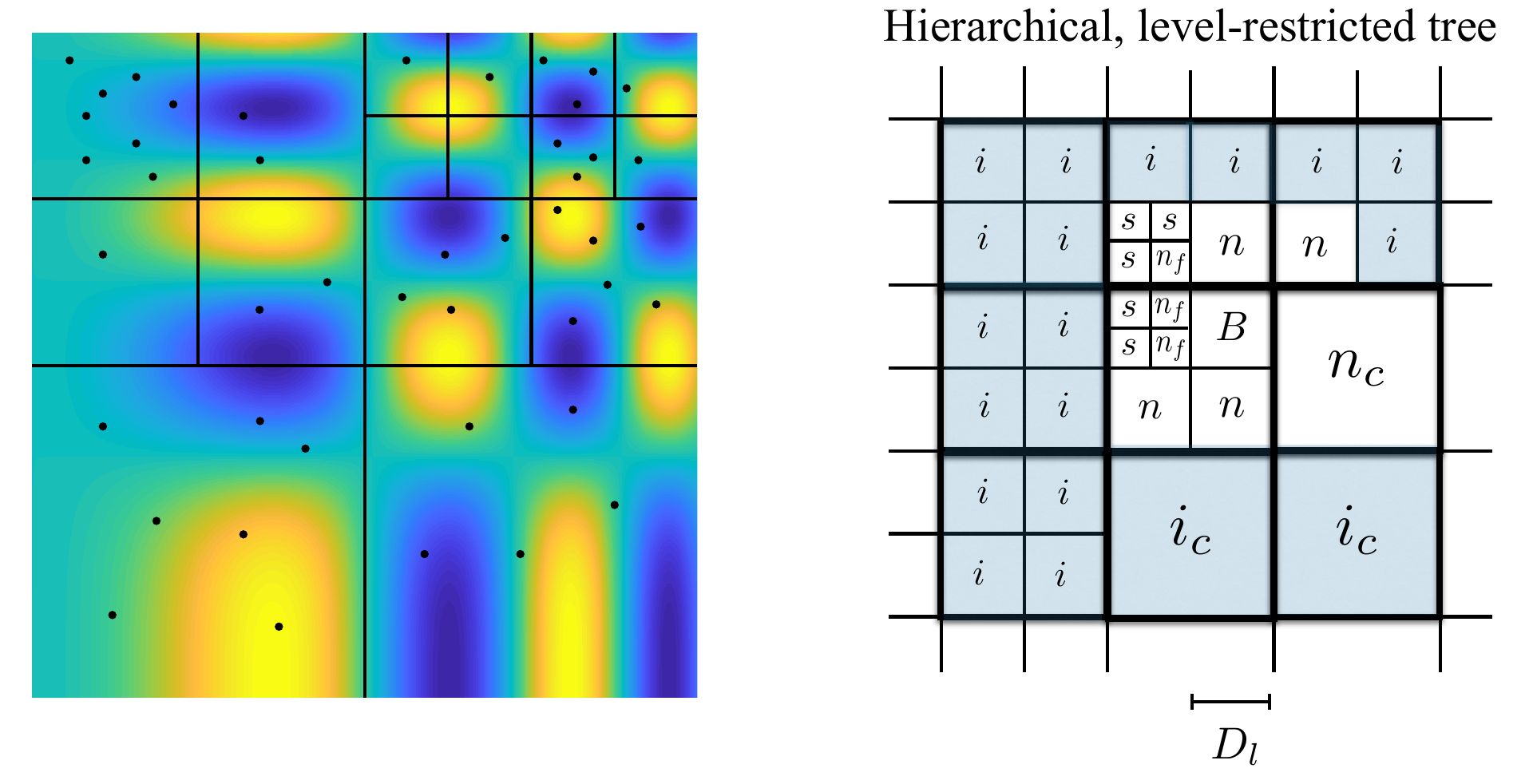}
\caption{The hierarchical FGT is able to handle
either volume or point sources discretized on a level-restricted
(balanced) quad-tree, discussed in section 
\ref{secfgt}.
On the left, we show a function and the corresponding resolving 
grid, together with a number of auxiliary point sources. On the right,
we indicate the kinds of interactions that must be accounted
for (see Definition \ref{ldef}).
Here, $B$ is a typical leaf node.
The boxes at the same level in the hierarchy that touch $B$ are 
called its {\em colleagues}, labeled $n$.
Denoting the parent of $B$ by $P$ (not shown),
the children of $P$'s colleagues
which do not touch $B$ define its interaction list (labeled $i$). 
In a level-restricted tree, a leaf node $B$
can have touching neighbors that are one level coarser,
labeled $n_c$, or one level finer, labeled $n_f$.
The boxes labeled $s$ are separated 
from $B$ but at a finer level, while the boxes
labeled $i_c$ are separated from $B$ but at a coarser level.}
\label{surfadap}
\end{figure}

\begin{remark}
In the original FGT \cite{greengard1991fgt}, 
a uniform grid is superimposed on the computational domain, with a 
box size of dimension $(r \sqrt{\delta})^d$, where $r \approx 1$.
Because of the exponential decay of the Gaussian, 
only nearby boxes need to be considered 
to achieve any desired precision and the corresponding field
is captured using a suitable
Hermite expansion (see section \ref{secfgt}).
For point sources, adaptivity is straightforward to achieve:
one simply assigns
source and target points to boxes on the uniform grid while
ignoring empty boxes, keeping track of the relevant neighbors 
for each box.
The total CPU time and storage is of the order $O(N+M)$.
(Some effort is required to do this without wasting storage.
One can, for example, refine a quad-tree uniformly 
to a level where the box dimension equals $(r \sqrt{\delta})^d$, 
pruning empty boxes on the way.)

Such a strategy cannot be used for initial potentials, however,
since there are typically no empty boxes. Rather, there is a continuous
function $u_0(\x)$ which has been specified on
a level-restricted adaptive quad-tree (which may have large regions where
the function is smooth but non-zero). 
\end{remark}

\subsection{Data structure} \label{sec:data} 

To be precise about our representation, following the presentation in
\cite{Ethridge2001sisc,wang2018sisc}, we assume that
the unit box $D$, centered at the origin contains the support
of $u_0(\x)$. We 
superimpose on $D$ a hierarchy of refinements using a {\em quad-tree}. 
For this, grid level 0 is defined to be $D$ itself, 
with grid level $l+1$ obtained recursively by subdividing each box at 
level $l$ into four equal parts. 
For any box $B$ at level $l$, the four boxes at level $l+1$ obtained 
by its subdivision will be referred to as its children.
Different subregions of $D$ may be refined to different depths.
For the sake of simplicity, we assume the tree structure
satisfies a standard condition: that 
two leaf nodes which share a boundary point must be no more than one 
refinement level apart. Such trees are generally called {\em level-restricted}
or {\em balanced} \cite{quadtree}.

On each leaf node $B$, we assume that we are given $u_0$ 
on a $k\times k$ tensor product Chebyshev
grid. From this data, we can construct
a ($k$-$1$)-degree polynomial 
approximation to $u$ on $B$ of the form
\begin{equation}
 u_B(\x) = u_B(x_1,x_2) \approx 
\sum_{j_1=0}^{k-1} \sum_{j_2=0}^{k-1} 
  c_B(j_1,j_2) \, T_{j_1}(x_1) T_{j_2}(x_2),
  \label{eq:leafnode}
\end{equation}
\noindent
where $T_j(x)$ denotes the Chebyshev polynomial of degree $j$ scaled
to the size of the box $B$.
In this approach, there are 
$k^2$ basis functions in the representation of $u_B(\x)$.

\begin{remark}
One could also use a polynomial basis
which satisfies the 
{\em total degree condition} $\{ j_1+j_2 \leq k-1 \}$
or the Euclidean truncation \cite{trefethen} 
$\{ j_1^2+j_2^2 \leq (k-1)^2 \}$. 
Both lead to a $k$-th order accurate
approximation as the mesh is refined. 
We will use the tensor product
approach since it is straightforward and
the coefficients can be computed 
efficiently using the cosine transform \cite{boyd2001}.
Finally, we note that for modest orders of accuracy, one could
use a uniform mesh on the leaf nodes instead of a Chebyshev mesh,
with a trivial modification to the method and code.
\end{remark}

Hierarchical fast algorithms have
been shown to be extremely effective on such data structures
in the elliptic setting
\cite{AskhamCerfon,cheng2006jcp,Ethridge2001sisc,langston2011camcs,GLee98,biros2015cicp,malhotra2016toms,wang2018sisc}.

\subsection{Adaptive refinement strategy}
\label{sec:funcapprox}

Let $B$ be a leaf node with $u_0(\x)$ given on that box by
a Chebyshev approximation of the form \eqref{eq:leafnode}.
One can then evaluate $u_B(\x)$ on a $2k \times 2k$ grid
covering $B$ and compute the discrete
$L_2$ error, denoted by $E_2$, over these target points.
If $E_2 > \epsilon_f$, for some user-specified tolerance $\epsilon_f$,
the leaf node $B$ is subdivided and the process repeated on each
of its children.  We will say that 
the tree obtained by a systematic use of this procedure 
starting with a single $k \times k$ grid on the unit box $D$
{\em resolves} $u_0(\x)$. However, it may not
be level-restricted. Assuming that 
the resolving tree so-constructed has $N$ leaf nodes and that its depth 
is of the order $O(\log N)$,
it is straightforward to balance the tree, enforcing the level-restriction,
in a subsequent sweep using $O(N \log N)$ time and storage
\cite{quadtree}. We omit the details.

One of the benefits of using a Chebyshev series in approximating
the function $u_0(\x)$ on a leaf node 
is that smoothness is manifested by rapid decay of the
coefficients with increasing $j_1,j_2$ in \eqref{eq:leafnode}.
Thus, one could also develop a heuristic for
building a resolving mesh based on the decay properties of the 
coefficients.  

\subsection{Hierarchical FGT} \label{secfgt}

Thorough descriptions of the FGT can be found in the references mentioned
above, namely 
\cite{greengard1991fgt,greengard1998nfgt,sampath2010pfgt,spivak2010sisc,tausch2009sisc}.
Here, we present only the main ideas underlying the hierarchical version
of \cite{wang2018sisc}. See also \cite{Lee2006}.

Let $h_n(x)$ denote the Hermite function defined by
\[ 
h_n(x)=(-1)^n \frac{d^n}{dx^n} e^{-x^2},\;\;\;x\in \bbR.
\]
Theese functions satisfy the relation
\begin{equation}
e^{-(x-y)^2/\delta}=\sum_{n=0}^{\infty} \frac{1}{n!}\left(\frac{y-y_0}{\sqrt{\delta}}\right)^n h_n\left(\frac{x-y_0}{\sqrt{\delta}}\right),
\label{mpole1d}
\end{equation}
where $y_0\in \bbR$ and $\delta>0$.
This Hermite expansion, centered at $y_0$, describes the 
Gaussian field $e^{-(x-y)^2/\delta}$ at the target $x$
due to the source at $y$.

In two dimensions, we make use
of multi-index notation. 
Letting $\x = (x_1,x_2) \in{\bf R}^2$, 
a multi-index is a pair of non-negative integers
$\alpha=(\alpha_1, \alpha_2)$ with the conventions:
\[
|\alpha| =\alpha_1+\alpha_2, \quad \alpha! =\alpha_1!\alpha_2!,  \quad
\x^{\alpha} =x_1^{\alpha_1} x_2^{\alpha_2}, \quad
D^{\alpha} =\partial_{x_1}^{\alpha_1}\partial_{x_2}^{\alpha_2}.
\]
If $p$ is an integer,
we say $\alpha \geq p$ if $\alpha_1,\alpha_2 \geq p$.
Multi-dimensional Hermite functions are defined by
\[
h_{\alpha}(\x)=h_{\alpha_1}(x_1)h_{\alpha_2}(x_2).
\]
The two-dimensional analog of
\eqref{mpole1d} is 
\begin{equation}
e^{-|\x-\y|^2/\delta}= \frac{1}{\alpha!}
\sum_{\alpha\geq 0} \left( \frac{\y-\y_0}{\sqrt{\delta}} \right)^{\alpha} 
h_{\alpha} \left( \frac{\x-\y_0}{\sqrt{\delta}} \right),
\label{mpole2d}
\end{equation}

The essential result from approximation theory used
in the FGT is the following.

\begin{lemma}\label{lemma:hexp} 
Let B be a box with center $\s_B$ and side length $r\sqrt{\delta}$ and let
the Gaussian field $\phi(\x)$ be defined by 
\begin{equation}  
\phi(\x)=\int_{B}e^{-\frac{|\x-\y|^2}{\delta}} f(\y)\, d\y \ +
\sum_{j=1}^{N_s} 
q_j e^{-\frac{|\x-\y_j|^2}{\delta}},
  \label{eq:volintb}
\end{equation}
where the $N_s$ source points $\{ \y_j \}$ lie in B.
Then,
\begin{equation}
\phi(\x)=
\sum_{\alpha\geq 0}A_{\alpha} h_{\alpha}\left(\frac{\x-\s_B}{\sqrt{\delta}}\right),
  \label{eq:hexp}
\end{equation}
where 
\begin{equation}
A_{\alpha}=
\frac{1}{\alpha!} \left( 
\int_{B} \left(\frac{\y-\s_B}{\sqrt{\delta}}\right)^{\alpha}f(\y)\, d\y + 
 \sum_{j=1}^{N_s} \left(\frac{\y_j-\s_B}{\sqrt{\delta}}\right)^{\alpha} q_j 
\right) \, .
  \label{eq:hcoeff}
\end{equation}
The error in truncating the Hermite expansion with $p^2$ terms is given by
\begin{equation}
 |E_H(p)|=\left|\sum_{\alpha\geq p}A_{\alpha} 
h_{\alpha}\left(\frac{\x-\s_B}{\sqrt{\delta}}\right)\right| \leq K^2Q_B(2S_r(p)+T_r(p))T_r(p),
  \label{eq:E_H(p)}
\end{equation}
where
\begin{equation}
Q_B=\int_B |f(\y)| d\y + 
\sum_{j=1}^{N_s} |q_j|,
  \label{eq:Q_B}
\end{equation}
\begin{equation}
S_r(p)=\sum_{n=0}^p \frac{r^n}{\sqrt{n!}},\;\;\;T_r(p)=\sum_{n=p}^{\infty} \frac{r^n}{\sqrt{n!}}.
  \label{eq:srtr}
\end{equation}
and $K < 1.09$.
\end{lemma}

For a proof, see \cite{wang2018sisc}.
The preceding formula alone is sufficient for developing
a fast algorithm. In the original FGT, 
after superimposing a uniform grid of boxes 
on the domain $D$
with box dimension $\sqrt{2 \delta}$, one proceeds in two steps.
First, one forms the Hermite expansion
induced by the sources in each box. 
The approximation error from truncating the Hermite expansions is 
controlled by eq.~\eqref{eq:E_H(p)}.
Second, for every target, one evaluates
the Hermite expansion due to the nearest $(2D+1)^2$ boxes, where
$D$ is a parameter determined by accuracy considerations. 
For this, note that the error incurred from ignoring more distant
interactions is of the order $O(e^{-4 D^2})$.
Thus, with $D=3$, double precision accuracy is achieved.
(The original FGT is a more elaborate scheme, but the complexity
result should be clear from this simple version.)

In the hierarchical FGT, 
we define a cut-off parameter $r_c$
so that 
$e^{-|\x-\y|^2/\delta}\leq \epsilon$, when $\| \x-\y\| \geq r_c\sqrt{\delta}$. 
We define the cut-off level in the FGT as the {\em finest} level where the box 
side length is greater than or equal to $r_c\sqrt{\delta}$.

\begin{definition} \label{ldef} 
[adapted from \cite{Ethridge2001sisc,wang2018sisc}]
In a level-restricted adaptive quad tree, for a leaf node $B$,
all nodes at the same level as $B$ 
which share a boundary point, including $B$ itself, are referred to as
{\em colleagues}. 
Leaf nodes at the level of 
$B$'s parent which share a boundary point with $B$ are referred to as its
{\em coarse neighbors}.
Leaf nodes one level finer than $B$
which share a boundary point with $B$ are referred to as its
{\em fine neighbors}.
Together, the union of the 
colleagues, coarse neighbors and fine neighbors of $B$
are referred to as $B$'s {\em neighbors}.
The {\em s-list} of a box $B$ consists of those children
of $B$'s colleagues which are not fine neighbors of $B$
(Fig. \ref{surfadap}).
The {\em interaction region} for $B$ consists of the area covered by
the neighbors of $B$'s parent, excluding the area covered by
$B$'s colleagues and coarse neighbors. The {\em interaction list}
for $B$ consists of those boxes in the interaction region which
are at the same refinement level (marked $i$ in Fig.\ref{surfadap}),
and is denoted by ${\cal I}(B)$.
Boxes at coarser levels will be referred to
as the {\em coarse interaction list}, denoted
by ${\cal I}_c(B)$
(marked $i_c$ in Fig.\ref{surfadap}).
Finally, if a source is in a box $B$ 
and a target lies outside $B$'s colleagues, then they are said to be 
{\em well-separated}.
\end{definition}

\begin{remark}
Let $B$ be a box in the quad-tree hierarchy 
with children denoted by $C_1,C_2,C_3,C_4$.
Then there is a linear operator ${\cal T}_{HH}$ which merges the 
expansions of four children 
into a single expansion for the parent. 

Since the Gaussian field induced by well-separated sources is
smooth, the field within $B$ induced by these sources
is well-represented by a Taylor series.
There is clearly a linear operator ${\cal T}_{LL}$ which
shifts such a Taylor series (a local expansion) 
from a parent box to its children.
Finally, 
for any box $F$ in $B$'s interaction list, 
there is a linear operator ${\cal T}_{HL}$ which converts
the Hermite expansion for box $F$ to its
induced local expansion in $B$.
\end{remark}

In the adaptive FGT,
leaf nodes need to 
handle far field interactions between boxes at different levels.
As illustrated in Fig. \ref{surfadap},
we need to incorporate the influence of the 
{\em s-list}  and 
the coarse interaction list on $B$.
For every box in the {\em s-list}, its Hermite expansion is rapidly 
convergent in $B$ and its 
influence can be computed by direct evaluation of the series. 
We also need to compute the dual interaction - namely the influence of
a leaf node $B$ on a box $F$ in the {\em s-list}.
We can
directly expand the influence of the source distribution in $B$
as a local expansion in $F$ from either point sources or a
continuous source distribution given in terms of a Chebyshev series.
The operator which maps the coefficients of the Chebyshev 
approximation of the density in a source box to the 
$p^2$ coefficients of the local expansion in a target box
can be precomputed and stored for each level in the quad-tree hierarchy.
Inspection of Fig. \ref{surfadap}, the translation invariance of the kernel,
separation of variables,
and a simple counting argument show that
this requires $O(k p L)$ work and storage, where 
$k$ is the order of polynomial approximation,
$p$ is the order of the local expansion, and 
$L$ is the number of levels.
We refer the reader to \cite{wang2018sisc} for a more 
thorough description of the scheme.

In the hierarchical FGT, as in the FMM,
two passes are executed. In the ``upward pass",
one first forms Hermite expansions 
using Lemma \ref{lemma:hexp} for all leaf nodes. Using the adaptive 
quad-tree data structure, for each interior node $P$,
one recursively merges the four Hermite expansions from $P$'s children
as a single expansion about the center of $P$.
The relevant translation operator that accomplishes this can be found
in \cite{wang2018sisc}.  In the ``downward pass", a {\rm Taylor expansion} 
for the root node (the domain $D$ itself) is initialized to zero.
Beginning at the root node, 
one shifts the Taylor expansion from each parent box at level $l$ 
to its children at level $l+1$. One then increments the local 
expansion for every box by adding in the contributions from all boxes 
in its interaction list. Finally,
for each leaf node, the local contributions are computed 
from the coarse neighbors, fine neighbors, $s$-list and coarse
interaction list, as sketched out
above. Precomputed tables make these local interactions especially 
efficient for continuous densities expanded in a Chebyshev series.
The total work required is of the order $O(N)$, where $N$ is the number
of grid points in all leaf nodes covering the domain $D$.

\subsection{Accelerations} \label{sec:accel}

Instead of Hermite expansions,
it was shown in \cite{greengard1998nfgt,spivak2010sisc} that an expansion in plane waves
provides additional acceleration. The essential observation is that
\begin{equation}
\sum_{\alpha\geq 0}A_{\alpha} h_{\alpha}\left(\frac{\x-\s_B}{\sqrt{\delta}}
\right) =
\int_{\mathbb{R}^2}
w(\k) e^{-\frac{\|\k\|^2}{4}} 
e^{i \k \cdot (\x - \s^B)/\sqrt{\delta}}  \,d\k
\, ,
\label{pwformula}
\end{equation}
where 
\begin{equation}
w(\k) = w(k_1,k_2) = \sum_{\alpha\geq 0} A_{\alpha} (-i)^{|\alpha|} 
k_1^{\alpha_1} k_{2}^{\alpha_2}. 
\end{equation}
This follows directly from the Fourier transform relation
\begin{equation}
e^{-\|\x\|^2} = \left(\frac{1}{4\pi}\right)
\int_{\mathbb{R}^2} e^{-\frac{\|\k\|^2}{4}}e^{i\k \cdot \x}\, d\k. 
\end{equation}
In practice, of course,
\eqref{pwformula} needs to be discretized. 
Because of the smoothness and exponential decay of the integrand, the 
trapezoidal rule is extremely effective. 
The advantage of \eqref{pwformula} over the Hermite representation
is that it provides a 
basis in which translation is diagonal.

The use of precomputed tables to accelerate the calculation of local
interactions is discussed in 
\cite{AskhamCerfon,cheng2006jcp,Ethridge2001sisc,langston2011camcs,biros2015cicp,GLee98,malhotra2016toms} 
for FMMs, and extended to the FGT in \cite{wang2018sisc}.
Many other accelerations are possible and we refer the reader to
\cite{sampath2010pfgt,spivak2010sisc,wang2018sisc} for further discussion.

\subsection{Periodic FGT}

The FGT is easily modified to handle periodic conditions
on the unit square $D = [-0.5,0.5]^2$.
Without entering into details, this is
accomplished by tiling the entire plane
${\bf R}^2$ with copies of the source distribution.
At the root node, this requires a lattice sum calculation to 
determine the Gaussian field in $D$ induced
by all well-separated copies of $D$. After that, it requires only
minor changes to the definitions of the neighbor and 
interaction lists to account for the nearest periodic images.
A complete discussion can be found in \cite{wang2018sisc}.

\subsection{Boundary FGT}

We turn now to the evaluation of {\em boundary Gauss transforms}
of the form 
\begin{equation}
\begin{split}
 {\cal S}[\sigma](\x) &=
\int_{\Gamma} e^{-\frac{|\x-\y(s)|^2}{\delta}} \tilde{\sigma}(\y(s)) \, ds_{\y} \, , \\
 {\cal D}[\mu](\x) &=
\int_{\Gamma} \frac{\partial}{\partial n_{\y(s)}} 
e^{-\frac{|\x-\y(s)|^2}{\delta}} \tilde{\mu}(\y(s)) \, ds_{\y}  \, .
\end{split}
\label{eq:bdryint}
\end{equation}
We assume that $\Gamma$ itself 
is described as the union of $M_b$ boundary segments:
\[ \Gamma = \cup_{j=1}^{M_b} \Gamma_j, \]
defined in terms of $k$th order Legendre polynomials in arc length: 
\begin{equation}
 \Gamma_j = \Gamma_j(s) = (x^1_j(s),x^2_j(s)): \ 
 x^1_j(s) = \sum_{n=0}^{k-1}  x^1_j(n) P_n(s), \ 
x^2_j(s) = \sum_{n=0}^{k-1}  x^2_j(n) P_n(s),  
\label{bdrydiscrete}
\end{equation}
with $-1 \leq s \leq 1$.
We also assume
that $\Gamma$ has been discretized in a manner that is 
commensurate with the underlying data structure used above:
an adaptive, level-restricted tree where the
length of $\Gamma_j$
is of approximately the same size as the 
box size of the leaf node that contains 
the center point $\c_j$ of $\Gamma_j$.
The densities $\sigma$ and $\mu$ in \eqref{eq:bdryint}
are assumed to be given in the form:
\[ \sigma_j(s) = \sum_{n=0}^{k-1}  \sigma_j(n) P_n(s), \quad
\mu_j(s) = \sum_{n=0}^{k-1} \mu_j(n) P_n(s).  \]

\begin{figure}[htbp]
\centering
\includegraphics[width=.4\textwidth]{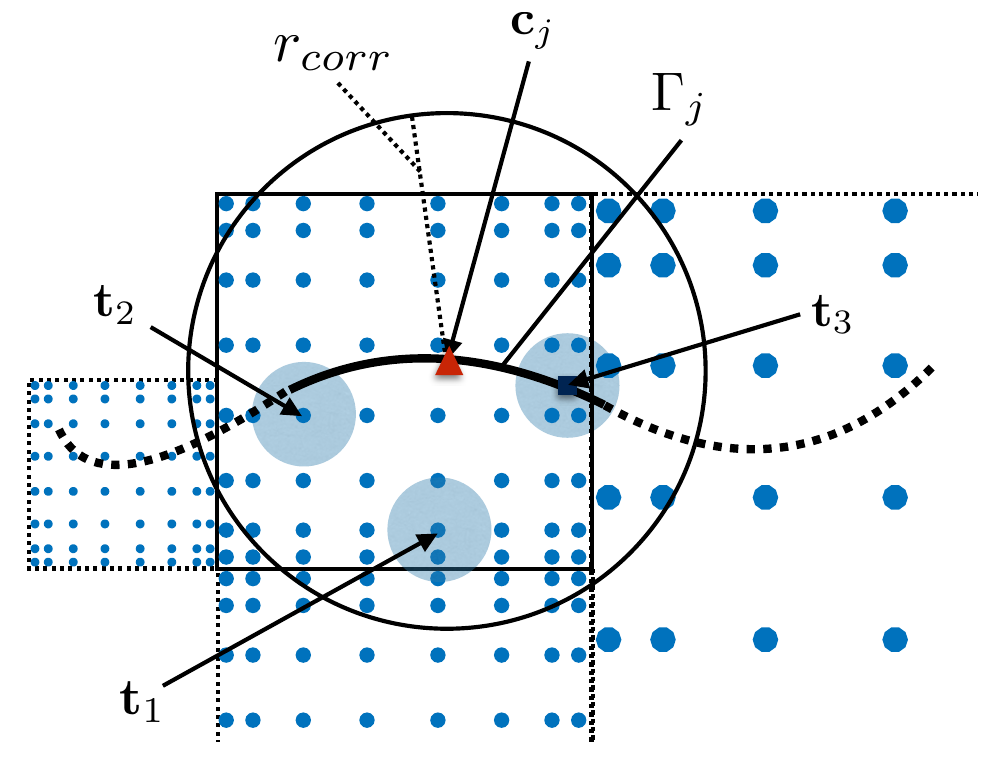}
\caption{
A boundary segment $\Gamma_j$ wth center $\c_j$ 
lying in a leaf node $B$ of side length $r_l$. 
Depending on the value of $\delta$, 
a boundary integral of the form  \eqref{eq:bdryint}
is either resolved by its discretization using standard Gauss-Legendre
quadrature with $k$ nodes on $\Gamma_j$, or negligible outside the disk
centered at $\c_j$ with radius $r_{corr}$.
In the latter case, when 
$\delta$ is small and the Gaussian is sharply peaked, 
a simple interpolatory rule can be used to 
compute the integral either on or off the boundary. 
The shaded disk in the figure around
each of the three target points $\t_i$ indicates the region where 
a Gaussian centered at $\t_i$ is less than a user-prescribed tolerance
$\epsilon$. Thus, 
the intersection of the shaded circles with $\Gamma_j$
indicates which subsegment of $\Gamma_j$ is relevant to the corresponding
target.
$\t_1$ is sufficiently far that its contribution can be
ignored. 
$\t_2$ and $\t_3$ are near and on the boundary, respectively.
Interpolation of the density from $\Gamma_j$ to the relevant
subsegment requires $O(k^2)$ work per target point
[Adapted from \cite{wang2018sisc}].
\label{bfgtfig}
}
\end{figure}

In the boundary FGT, suppose that we apply
composite Gauss-Legendre quadrature to the 
integrals in \eqref{eq:bdryint}. 
The accuracy of the quadrature depends strongly on the smoothness of
the integrand, and therefore 
on the parameter $\delta$. Fortunately, the rapid decay of the
Gaussian makes the problem easily tractable for any $\delta$.
To see why, consider a boundary segment $\Gamma_j$, centered at $\c_j$ 
in a leaf box $B$ of commensurate size (Fig. \ref{bfgtfig}). 
If $\delta$ is sufficiently large, say 
${\delta} > |\Gamma_j|^2$, then
spectral accuracy is achieved with no further correction.
If, on the other hand, 
${\delta} < |\Gamma_j|^2$, then the integral needs to 
be corrected only within some disk of radius
$r_{corr} =  O(|\Gamma_j|)$.
This correction can be computed rapidly and accurately by 
determining the subinterval of $\Gamma_j$ that makes a non-negligible 
contribution to the integral and interpolating the original density
to $k_c$ Gauss-Legendre nodes on this subinterval.
Setting $k_c$ to 20 yields approximately fourteen digits of accuracy
assuming the density $\sigma(s)$ is locally smooth.
We refer the reader to \cite{wang2018sisc} for further details.

\subsection{Solving the free-space heat equation by marching} \label{sec:march}

Suppose now that we wish to evaluate the solution to the initial
value problem in free space with compactly supported initial data
$u_0(\x)$ via marching, according to \eqref{vmarch}, repeated here 
for convenience:
\[
\begin{aligned}
u^{(V)}(\x,t+\delta) = 
&\int_{\mathbb{R}^2} G(\x-\y,\delta) \, 
u^{(V)}(\y,t) \, d\y + \\
&\int_{t}^{t+\delta}
\int_{\Omega(\tau)} G(\x-\y,t+\delta-\tau) \, F(\y,\tau) \, d\y d\tau.
\end{aligned}
\]

This requires that we extend
our spatial data structure as time progresses
beyond the initial domain $\Omega(0)$ to capture the spread of
$u^{(V)}(\y,t)$ with time.
Moreover, to obtain a viable numerical method, 
the number of additional degrees of freedom should grow slowly
with time.
Fortunately, doing so is quite straightforward (as observed 
previously in \cite{strain1994sisc}) and requires only the
following Lemma
(see Fig. \ref{fig:freespace_grid}).

\begin{lemma} \label{thm:resolve}
\cite{wang2017thesis}
  Suppose that a collection of sources is contained
  in a box $D$ in $\bbR^d$ and that the box $B$ of side length $R$
  is {\it well-separated} 
  from $D$ - that is separated from $D$ by at least $R$,
  For any $\epsilon <1$, let $\phi(\x,t)$ denote the 
  field induced in $B$ by the sources in $D$.
  Then a Taylor series approximation
with $p=\mathcal{O}\left(\log\left(\frac{1}{\epsilon}\right)\right)$ terms
is sufficient to resolve $\phi(\x,t)$ to precision $\epsilon$
for all $t>0$.
\end{lemma}

\begin{proof}
  We consider the case of a volumetric source.
  The same proof applies in the case of layer potentials and initial
  potentials.
  Thus, suppose
  \be\label{slp}
  \phi(\x,t)
  \int_0^t\int_{D} \frac{1}{4\pi(t-\tau)}e^{-\frac{\|\x-\y\|^2}{4(t-\tau)}}
  F(\y,\tau) \, d{\y} \, d\tau.
  \ee
  Note, first, that the Green's function is less than $\epsilon$
  for $R^2 > 1/\epsilon$. 
  Since $B$ is {\it well-separated}, we have
  \be\label{resolve1}
  \|\x-\y\| \ge R, \quad \text{for}\,\, \x\in B,\, \y\in \Gamma.
  \ee
  Thus,
  \be\label{resolve2}
  G(\x-\y,t-\tau)\le \frac{1}{4\pi(t-\tau)}e^{-\frac{R^2}{4(t-\tau)}},
  \quad \text{for}\,\, \x\in B,\, \y\in \Gamma, \, \tau<t.
  \ee
  Assuming $R^2 < 1/\epsilon$, 
  the right side of \eqref{resolve2} is less than $\epsilon$ when
  \be\label{resolve3}
  t-\tau<\frac{R^2}{2\log\left(\frac{1}{\epsilon R^2}\right)}.
  \ee
  Letting \be
  C=\frac{1}{2\log \left(\frac{1}{\epsilon R^2}\right)},
  \ee  
  we have that
  \be
  \phi(\x,t) = \int_0^{t- C R^2} \int_{D} 
  G(\x-\y,t-\tau) F(\y,\tau) \, d{\y} \, d\tau \ + \mathcal{O}(\epsilon)
  \ee
  assuming that $t>C R^2$.
  It has been shown, however (\cite[eq. 9]{strain1991vfgt} and 
  \cite[eq. 16]{wan2006jcp}) that 
  a Taylor series of order $p$ is sufficient
  to resolve the heat kernel in $B$ for any $\tau \in [0,t-R^2 C]$ with an 
  error of the order 
  \be
  |E_p| \leq 2Kp^{-1/4}\frac{r_p^p}{1-r_p},
  \ee
  where $K=1.09(2\pi)^{-1/4}$ and 
  \be
  r_p=\frac{R}{\sqrt{2C R^2}}\sqrt{\frac{e}{p}}=\sqrt{\frac{e}{2p C}}
  =\sqrt{\frac{e \log\left(\frac{1}{\epsilon R^2}\right)}{p}}
  <\sqrt{\frac{e(2n+1) \log\left(\frac{1}{\epsilon}\right)}{p}}.
  \ee
  This converges extremely rapidly with $p$ once $r_p<1$ (i.e.,
  $p>e(2n+1)\log\left(\frac{1}{\epsilon}\right)$).
  The result follows immediately with a net error in the 
  Taylor series bounded by
  \[ |E_p| \| F \|_1, \] 
  where
  \[ \| F \|_1 = \int_0^{t- R^2 C} \int_{D} 
  |F(\y,\tau)| \, d{\y} \, d\tau .
  \]
\end{proof}

\begin{figure}[htbp]
\centering
\includegraphics[width=.8\textwidth]{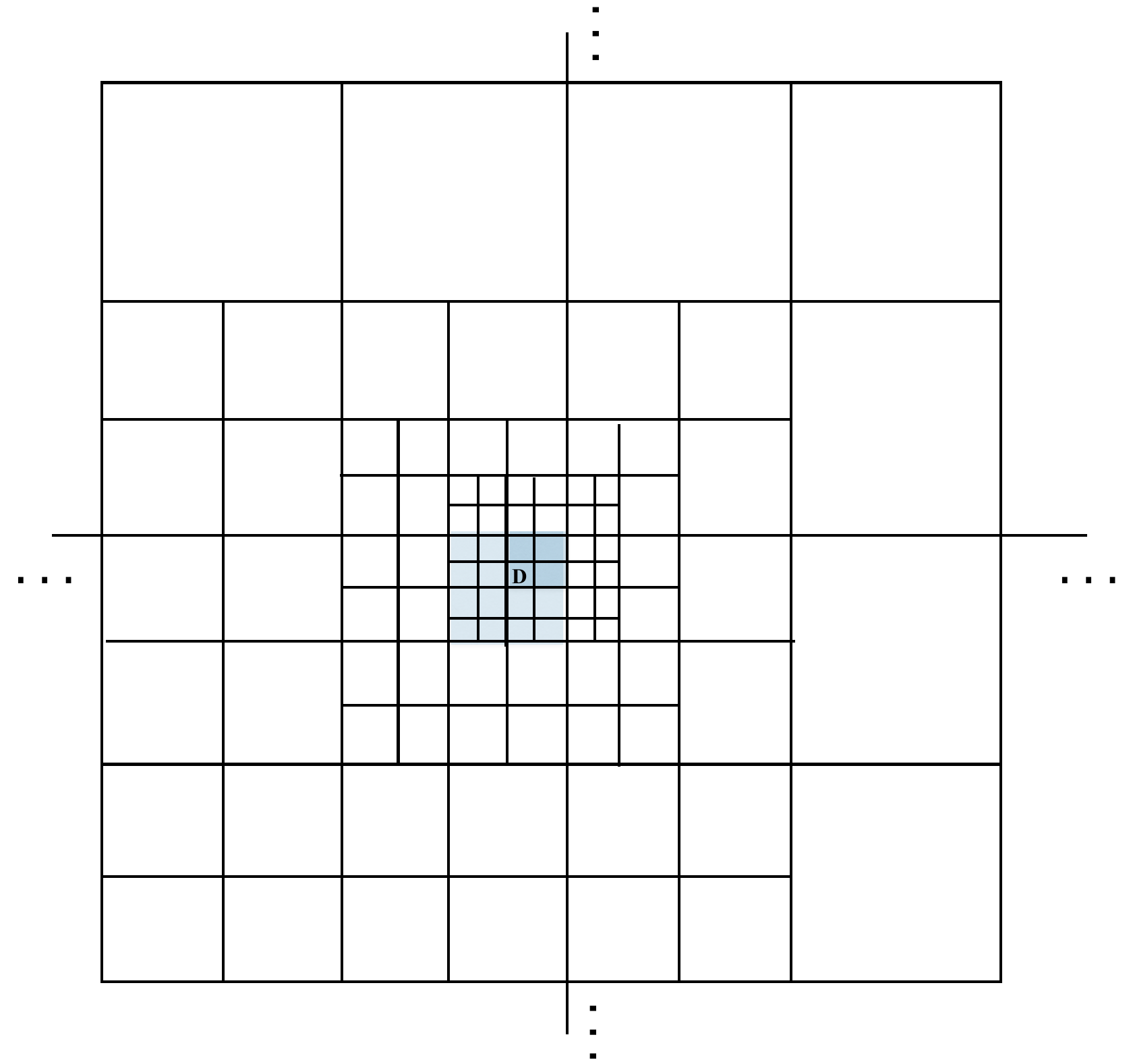}
\caption{The hierarchical FGT is able to resolve the solution
of the heat equation in free space using a level-restricted
quad tree by extending the coverage of $\mathbb{R}^2$ from the 
domain $D$ with side length $\rho$ containing the sources to a domain of size 
approximately $\rho + \sqrt{T\log(T/\epsilon)}$.
This captures the entirety of the 
solution out to time $T$ with error of the order $\epsilon$, as noted in
\cite{strain1994sisc}.
Because of Lemma \ref{thm:resolve}, the total number of additional 
grid points
needed is approximately $O((\log T + \log \epsilon)^2)$.
Thus, capturing the solution in free space adds remarkably little to the
cost in terms of CPU time or storage.
}
\label{fig:freespace_grid}
\end{figure}

\begin{remark} \label{ndegrmk}
Within the box $D$, the number of spatial degrees of freedom, denoted by $N$,
is determined by the number of degrees of freedom needed to resolve the 
data $F(\x,t)$. Subsequent convolution with the heat kernel smooths 
out this data over time. Thus, in the 
absence of subsequent volumetric sources, 
the number of degrees of freedom needed within $D$ decreases 
because of the smoothing properties of the heat kernel. 
From Lemma \ref{thm:resolve}, it can be shown that the 
number of additional grid points needed 
to capture the solution diffusing out into 
free space grows with the final time $T$ at the rate
$O((\log T + \log \epsilon)^2)$.
(see Fig. \ref{fig:freespace_grid}).
\end{remark}

\begin{remark} \label{erroraccum}
One difficulty with marching according to \eqref{vmarch}
is that worst case error analysis suggests we are
committing an error of the order $\epsilon$ at each time step. Thus, the
total error that could be as much as $N_T \epsilon$ for $N_T$ time steps.
As noted in \cite{strain1994sisc}, however,
the errors are 
``high frequency" and rapidly decaying in time,
so that repeated use of the Gauss transform 
does not result in such a severe accumulation of error in practice.
(Spectral fast algorithms, described below, avoid this potential 
accumulation of error.)
\end{remark}

\subsection{The rapid evaluation of layer potentials using the FGT}
\label{sec:layerhist}

We turn now to the evaluation of single and double layer potentials
on a space-time boundary $\Gamma_T$ contained within a region $D$,
focusing on the history part 
$\cs_H[\sigma](\x,t)$. (The double layer contribution
$\cd_H[\mu](\x,t)$ is handled in the same manner.)
From the semigroup property of the heat kernel, 
\be\label{semigroupproperty}
G(\x-\y,t+\delta-\tau)=\int_{\mathbb{R}^d}G(\x-\z,\delta)G(\z-\y,t-\tau)d\z,
\ee
it is straightforward to 
see that
\be
\ba
\cs_H[\sigma](\x,t+\delta)&=
\int _{0}^{t} \int _{\Gamma(\tau)}
G(\x-\y,t+\delta-\tau) \;\sigma(\y,\tau) \; ds_{\y} \;d\tau \,  \\
&=\int_{\mathbb{R}^d}G(\x-\y,\delta)\cs[\sigma](\y,t) d\y.
\ea
\ee
In other words, the history part of the single layer potential
at time $t+\delta$ is simply the initial
potential acting on the data
$\cs[\sigma](\x,t)$ for a time step $\delta$. 
This gives rise to two questions:
(1) how is $\cs[\sigma](\x,t)$ to be resolved on a spatial mesh 
and (2) how will we evaluate the local part $\cs_L[\sigma](\x,t)$.
The latter question is discussed below in section
\ref{sec:loc}. As for the first question, it should be noted that
$\cs[\sigma](\x,t)$ is smooth inside and outside of $\Omega(t)$, assuming 
the domain is smooth, but it is not smooth {\em across} the boundary.
While it is possible for the adaptive FGT to handle such nonsmooth data,
it is more convenient in practice to modify the algorithm slightly
to avoid this issue.
The FGT, like volume integral versions of the FMM, is much faster when
it acts on data that is resolved on an adaptive spatial mesh as a
piecewise polynomial because
of the ability to use precomputed tables of local interactions, 
as mentioned in section \ref{sec:accel}.

\subsubsection{The bootstrapping method} \label{sec:bootstrap}

To enable working only with smooth data,
we further decompose the history part into two parts -- a 
{\em near history} part and a {\em far history} part.

\begin{equation}
\label{histdecomp}
\begin{split}
  \cs_H[\sigma](\x,t)
  &=
\cs_{NH}[\sigma](\x,t) + \cs_{FH}[\sigma](\x,t), \\
\cd_H[\mu](\x,t) &=
\cd_{NH}[\mu](\x,t) + \cd_{FH}[\mu](\x,t), 
\end{split}
\end{equation}
where
\begin{equation}
\label{bootstrap}
\begin{split}
\cs_{NH}[\sigma](\x,t) &=
\int _{t-2\delta}^{t-\delta} \int _{\Gamma(\tau)}
G(\x-\y,t-\tau) \;\sigma(\y,\tau) \; ds_{\y} \;d\tau \, , \\
\cd_{NH}[\mu](\x,t) &=
\int_{t-2\delta}^{t-\delta} \int_{\Gamma(\tau)} 
\frac{\partial}{\partial \nu_{\y}} G(\x-\y,t-\tau) \mu(\y,\tau) ds_{\y}d\tau\, ,
\end{split}
\end{equation}
and
\begin{equation}
\label{smooth}
\begin{split}
\cs_{FH}[\sigma](\x,t) &=
\int _{0}^{t-2\delta} \int _{\Gamma(\tau)} G(\x-\y,t-\tau) \;
\sigma(\y,\tau) \; ds_{\y} \;d\tau \, , \\
\cd_{FH}[\mu](\x,t) &=
\int_0^{t-2\delta} \int_{\Gamma(\tau)} 
\frac{\partial}{\partial \nu_{\y}} G(\x-\y,t-\tau) \mu(\y,\tau) ds_{\y}d\tau\, .
\end{split}
\end{equation}
In the bootstrapping version of the recurrence,
the near history is computed by quadrature and incorporated
into the local part in section \ref{sec:loc}.
For the far history, we have the following lemma, showing how to 
propagate the far history as a recurrence for a function tabulated
on an adaptive spatial mesh.

\begin{lemma}\label{initrep}
  Let 
\[ u_{FH}(\x,t+\delta) = \cs_{FH}[\sigma](\x,t+\delta), \quad
  v_{FH}(\x,t+\delta) = \cd_{FH}[\mu](\x,t+\delta).
\]
  Then
  \be\label{sinitrep}
  u_{FH}(\x,t+\delta) = 
  \I\left[ u_{FH}(\cdot,t)+\cs_{NH}[\sigma](\cdot,t)\right] (\x,\delta) \,,
  \ee
  and
  \be\label{dinitrep}
  v_{FH}(\x,t+\delta) = 
  \I\left[ v_{FH}(\cdot,t)+\cd_{NH}[\mu](\cdot,t)\right] (\x,\delta) \, .
  \ee
  Here, the domain of integration for the initial potential is $\bbR^2$
  or, to precision $\epsilon$, the extended domain of size
  $|D| + \sqrt{T \log(T/\epsilon)}$, as discussed in section 
  \ref{sec:march}.
  The distinction between $u_{FH}, v_{FH}$ and 
  $\cs_{FH}[\sigma], \cd_{FH}[\mu]$
  is that the former are functions sampled on a spatial mesh, while
  the latter are layer potentials defined in terms of the densities
  $\sigma$ and $\mu$.
\end{lemma}
\begin{proof}
  We restrict our attention to \eqref{sinitrep}, since
  the proof of the second identity is analogous.
  By direct calculation, we have
  \[
  \ba
  u_{FH}(\x,t+\delta) &= 
\int _{0}^{t-\delta} \int _{\Gamma(\tau)}
G(\x-\y,t+\delta-\tau) \;\sigma(\y,\tau) \; ds_{\y} \;d\tau \,  \\
&=\int _{0}^{t-\delta} \int _{\Gamma(\tau)}
\left(\int_{\mathbb{R}^2}G(\x-\z,\delta)G(\z-\y,t-\tau)d\z\right)
\sigma(\y,\tau) \; ds_{\y} \;d\tau \, \\
&=\int_{\mathbb{R}^2}G(\x-\z,\delta)
\left(\int _{0}^{t-\delta} \int _{\Gamma(\tau)} G(\z-\y,t-\tau)
\sigma(\y,\tau) \; ds_{\y} \;d\tau\right)d\z\\
&=\int_{\mathbb{R}^2}G(\x-\z,\delta)\cs_H[\sigma](\z,t) d\z\\
&=\int_{\mathbb{R}^2}G(\x-\z,\delta)\left(\cs_{FH}[\sigma](\z,t)+\cs_{NH}[\sigma](\z,t)\right) d\z\\
&=\I\left[u_{FH}(\cdot,t)+\cs_{NH}[\sigma](\cdot,t)\right](\x,\delta).
  \ea
  \]
  The first equality follows from the definition of $u_{FH}(\x,t+\delta)$,
  the second follows from \eqref{semigroupproperty}, the third follows from
  the change of integration order, and the last follows from the definition
  of $\cs_H$.
\end{proof}
In short, the {\em far history} at time $t+\delta$ is the initial potential
acting on data that corresponds to the far history at time $t$ 
incremented by the near history at time $t$. The latter
``bootstrap" contribution has evolved under heat flow
for a time at least $\delta$, so that it is a smooth function
of $\x$ at time $t$. It is evaluated on the adaptive spatial grid at time
$t$ by the methods discussed in the next section.

In the bootstrapping framework, we can be more precise about how many
spatial degrees of freedom are needed to resolve the far history part 
of the layer potential as well as the near history increment.
A remarkable fact is that the number of points needed
for resolution in a domain $\Omega$ is {\em of the same order as the
number of points on the boundary}. This follows from
Lemma \ref{thm:resolve}, illustrated when
the domain is a square with a uniformly discretized boundary
(Fig. \ref{countadap}).
For this case, as shown in \cite{wang2017thesis},
we suppose that there are $3 \cdot 2^L$ 
subintervals on each side and $p$ 
points on each interval. Then the number of 
points used to discretize $\Gamma$ is about $N = 12p \cdot 2^L$.
Each box in the adaptive discretization of the domain
(excluding the boxes that touch the boundary) satisfies
the condition of 
Lemma \ref{thm:resolve}. 
At level $L$, there are approximately $3 \cdot 2^L \cdot 2 \cdot 4=2N/p$
boxes. For coarser levels, the number of boxes decreases at least by
a factor of $2$, since the size of the box increases by a factor of $2$. 
Thus, the total number of boxes is bounded by
$2N/p \cdot \left(1+\frac{1}{2}+\frac{1}{4}+\ldots\right) = 4N/p$. 
With a $p \times p$ subgrid in every box, the total number of interior 
points is bounded by $4N/p \times p^2 = 4pN$.
  
\begin{figure}[thbp]
\centering
\includegraphics[width=.4\textwidth]{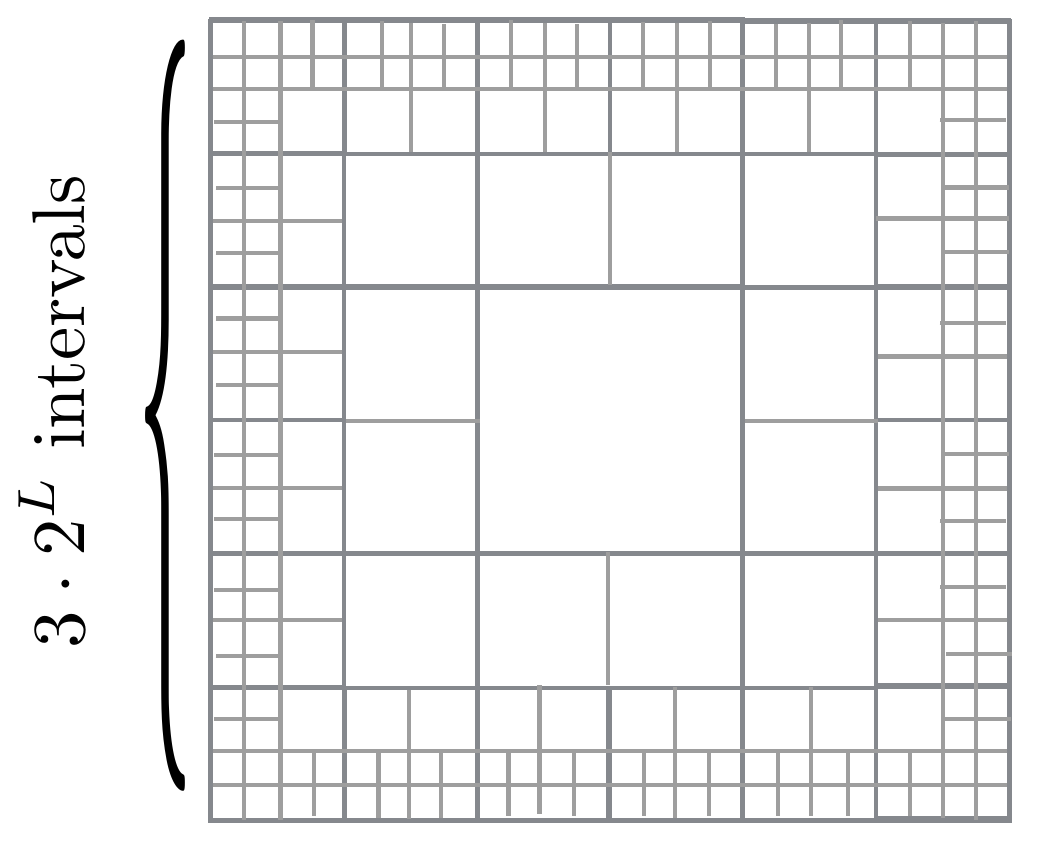}
\caption{Number of interior boxes needed
to resolve a layer potential when the boundary
is a square with each side divided into $3\cdot 2^L$ equal subintervals.}
\label{countadap}
\end{figure}

As noted in section \ref{sec:march},
 the number of degrees of freedom needed to capture the history
part in the exterior of the bounding box $D$
is no more than an additional   
$O((\log T + \log \epsilon)^2)$ grid points.

\begin{remark}
Since the FGT has linear computational complexity.
the cost of capturing the history part for all time steps
is of the order
$\mathcal{O}(N_SN_T\log\left(\frac{1}{\epsilon}\right))$.
As a result, it is possible to solve {\em homogeneous} boundary value 
problems in optimal time using
the bootstrapping scheme together with an adaptive
data structure and a suitable fast algorithm.
\end{remark}

\begin{remark}
A version of the bootstrapping scheme was described earlier
\cite{brattkus1992siap}, but without spatial adaptivity
(or a hierarchical fast algorithm).
Using an adaptive triangulation to capture the history part
of a layer potential was described previously in 
\cite{strain1994sisc}, but without bootstrapping.
\end{remark}

\subsubsection{Evaluation of the local part of heat potentials} \label{sec:loc}

While the history part of heat potentials dominates the cost of their
evaulation without suitable fast algorithms, a careful treatment of the 
local part of layer potentials (including the near history part)
is needed to handle the singular behavior of the Green's function
for short times. 
Volume heat potentials, it turns out, are much simpler to deal with.
The local part of a volume potential
\[
\V_L[F](\x,t) =
\int_{t-\delta}^t \int_{\Omega(\tau)} 
G(\x-\y,t-\tau) \; F(\y,\tau) d{\y}d\tau
\]
involves an inner integral of the form
which is a mollification of $F(\y,\tau)$. Thus, assuming 
$F$ is itself a smooth function, a $k$th order quadrature rule in time
yields $k$th order accuracy. As a simple example, the trapezoidal rule
takes the form:
\[
\V_L[F](\x,t) =
\frac{\delta}{2} \left(
F(\x,t) + \int_{\Omega(t-\delta)} 
G(\x-\y,\delta) \; F(\y,t-\delta) d{\y} \right)
+ O(\delta^3),
\]
where we have made use of the ``$\delta$-function" property of the 
heat kernel in the 
upper limit of integration $\tau \rightarrow t$,
as discussed in \cite{epperson2,strain1994sisc}.

For layer potentials, there are numerous approaches that have
been developed over the years to handle the local quadrature issues.
They include asymptotic methods, Gauss-Jacobi integration,
``full" product integration, and hybrid asymptotic/numerical methods.
Asymptotic methods are the easiest to use and various
expansions can be found in the literature for layer potentials
on surface \cite{greengard1990cpam,lin1993thesis,li2009sisc} or
off-surface. The following result is taken from \cite{wang2019acom}.
Related formulae can be found in \cite{strain1994sisc}.

\begin{lemma}\label{lemma:locasym}
  Let $\Gamma(\tau)$, $\sigma$ and $\mu$ be four times differentiable.
  Then
\begin{equation}\label{eq:slocasym}
\cs_\delta[\sigma](\x,t)=\frac{1}{2}
\sqrt{\frac{\delta}{\pi}}E_{3/2}\left(\frac{c^2}{4}\right)
\left(1+\frac{\kappa-v}{2}\cdot c\sqrt{\delta}\right)
\cdot\sigma(\x_0,t)+O(\delta^{3/2}),
\end{equation}
\begin{equation}\label{eq:dlocasym}
\begin{split}
\cd_\delta[\mu](\x,t)=&-\sqrt{\frac{\delta}{\pi}}E_{3/2}\left(\frac{c^2}{4}\right)\frac{\kappa-v}{4}\mu(\x_0,t) \\
&-\frac{\sgn(c)}{2}\erfc\left(\frac{|c|}{2}\right)(1+c\sqrt{\delta}\cdot\frac{\kappa-v}{2})\mu(\x_0,t)+O(\delta),
\end{split}
\end{equation}
where $\x=\x_0+c\sqrt{\delta}\cdot\n$, $\x_0\in\Gamma(t)$ is a point on 
the boundary at time t, $\n$ is the inward unit normal vector 
at $\x_0$, $\kappa$ is the curvature at $\x_0$, and $v$ is the normal 
component of the velocity of $\Gamma$ at $\x_0$. 
Here, $E_{3/2}$ denotes the exponential integral function of order 3/2:
\begin{equation}
E_{3/2}(x)=\int_1^{\infty}\frac{e^{-xt}}{t^{3/2}}dt \, 
\end{equation}
and $\erfc(x)=\frac{2}{\sqrt{\pi}}\int_x^{\infty} e^{-t^2} dt$ is the 
complementary error function.
\end{lemma}

The on-surface asymptotics in
\cite{greengard1990cpam,lin1993thesis,li2009sisc} may be recovered by taking
the limit $\x\rightarrow \x_0$ (i.e. $c\rightarrow 0$) in \eqref{eq:slocasym}
and \eqref{eq:dlocasym}.

\begin{corollary}\label{cor:slocasym0}
Let $\Gamma(\tau)$, $\sigma$ and $\mu$ be four times differentiable,
and $\x_0\in\Gamma(t)$ be a point on the boundary at time t. Then
\begin{align}
\cs_\epsilon[\sigma](\x_0,t)&=\sqrt{\frac{\epsilon}{\pi}}\sigma(\x_0,t)\nonumber \\ 
&+\epsilon^{3/2}\left(\frac{(\kappa-v)^2\sigma(\x_0,t)}{12\sqrt{\pi}}
+\frac{\sigma_t-\sigma_{ss}}{3\sqrt{\pi}}\right)+O(\epsilon^{5/2}), \label{eq:slocasym0} \\ 
\cd_\epsilon[\mu](\x_0,t)&=-\sqrt{\frac{\epsilon}{\pi}}
\frac{\kappa-v}{2}\mu(\x_0,t)+O(\epsilon^{3/2}), \label{eq:dlocasym0}
\end{align}
where $s$ is the arc length parameter.
\end{corollary}

Such expansions may be carried out to higher order,
but the expressions are rather complicated and involve higher derivatives of the
densities and/or boundary. Moreover,
it was shown in \cite{li2009sisc} that the error doesn't 
match the estimates in Lemma \ref{lemma:locasym} until $\delta$ is of 
the order $O(\Delta x^2)$, so that a purely asymptotic approach is not
robust. This phenomenon was referred to in \cite{li2009sisc}
as {\em geometrically-induced stiffness}.
Likewise, rewriting the local part of the single layer potential as
\[
\cs_L[\sigma](\x,t) =
\int _{t-\delta}^t 
\frac{1}{\sqrt{ 4 \pi (t-\tau)}} \int _{\Gamma(\tau)} 
\frac{e^{- \| \x-\y\|^2/( 4(t-\tau))}}{\sqrt{4 \pi (t-\tau)}}
\sigma(\y,\tau) \; ds_{\y} \;d\tau ,
\]
it would appear that the use of a quadrature rule that handles the 
$\frac{1}{\sqrt{ 4 \pi (t-\tau)}}$ singularity in the outer integral
would be very effective, 
since the inner integral can be shown to be smooth as a function of time.
High-order accuracy, for example, can be achieved using 
Gauss-Jacobi quadrature \cite{DR}
or Alpert's hybrid Gauss-trapezoidal rule \cite{alpert1999sisc} for 
an inverse square root singularity.
Unfortunately, it was shown in \cite{li2009sisc} that this method, too,
suffers from geometrically-induced stiffness and the formal order of 
accuracy isn;t seen until the time step is of 
the order $O(\Delta x^2)$.
A method referred to as 
``full product integration" was suggested as a remedy.
For this, the density is expanded at each boundary point as a Taylor
series in time
\[ 
\sigma(\x,\tau) = 
\sigma_0(\x) + (t-\tau) \sigma_1(\x) + \frac{1}{2} (t-\tau)^2
\sigma_2(\x) + \dots
\]
For fixed geometries,
exchanging the order of integration with respect to space and time,
the single layer potential $\cs_L$ takes the form
\begin{align*}
\cs_L[\sigma](\x,t) =  \frac{1}{4 \pi} \bigg[
&\int_{\Gamma} G_0(\x-\y) \sigma_0(\y) \, ds_{\y} 
+ 
\int_{\Gamma} G_1(\x-\y) \sigma_1(\y) \, ds_{\y} 
 + \dots \\
 &+
\frac{1}{(k-1)!} 
\int_{\Gamma} G_{k-1}(\x-\y) \sigma_{k-1}(\y) \, ds_{\y} \bigg]
 + O( (t-\tau)^{k+1/2}) ,
\end{align*}
with 
\begin{equation}
\label{time_kernels}
G_k(\x) =
\int_{t-\delta}^t 
e^{- \| \x \|^2/4(t-\tau)} (t-\tau)^{k-1} \, d\tau . 
\end{equation}
Thus, one can approximate the local part of a layer potential
with high order accuracy
as a sum of spatial integral operators acting on 
$\sigma_0(\x), \sigma_1(\x),\dots,\sigma_{k-1}(\x)$.
The kernels $G_k(\x)$ can be obtained in closed form and have
logarithmically singular kernels. 
For moving geometries, a slightly more complicated scheme is required
but, in either case,  geometrically-induced stiffness is avoided.
This method was implemented
in \cite{wang2016njit,wang2019jsc},
using recursive skeletonization to rapidly apply the 
(non-oscillatory) spatial integral operators 
\cite{ho2012sisc,ying2017siammms}.
Many other herarchical, fast algorithms
are available that reduce the cost from $\mathcal{O}(N_S^2)$
to  $\mathcal{O}(N_S)$ or $\mathcal{O}(N_S \log N_S)$, where $N_S$ denotes
the number of boundary points
(e.g., \cite{HSS,cheng1999jcp,fong2009jcp,gimbutas2003sisc,greengard1987jcp,greengard1997actanum,Hackbusch2015,ho2016cpam1,ho2016cpam2,ying2017rms,ying2004jcp,zhang2011jcp}).

Recently, a hybrid scheme was developed for evaluating the local part
of layer potentials, combining asymptotics and numerical quadrature
after a change of variables \cite{wang2017thesis,wang2019acom}.
First, one further splits the local part into a singular term and a 
near singular term:
\begin{equation}\label{slocalsplit}
  \cs_L[\sigma](\x,t) = \cs_\epsilon[\sigma](\x,t) + \cs_{L^*}[\sigma](\x,t)
\end{equation}
and
\begin{equation}\label{dlocalsplit}
\cd_L[\mu](\x,t) = \cd_\epsilon[\mu](\x,t) + \cd_{L^*}[\mu](\x,t)
\end{equation}
where
\[
\begin{aligned}
\cs_\epsilon[\sigma](\x,t) &:=\int_{t-\epsilon}^t \int_{\Gamma(\tau)}
G(\x-\y,t-\tau) \;\sigma(\y,\tau) \; ds_{\y} \;d\tau \, , \\
\cs_{L^*}[\sigma](\x,t) &:= \int_{t-\delta}^{t-\epsilon} \int_{\Gamma(\tau)}
G(\x-\y,t-\tau) \;\sigma(\y,\tau) \; ds_{\y} \;d\tau \, ,  \\
\cd_\epsilon[\mu](\x,t) &:=\int_{t-\epsilon}^t \int_{\Gamma(\tau)}
\frac{\partial}{\partial \nu_{\y}} G(\x-\y,t-\tau) \mu(\y,\tau) ds_{\y}d\tau\, ,\\
\cd_{L^*}[\mu](\x,t) &:= \int_{t-\delta}^{t-\epsilon} \int_{\Gamma(\tau)}
\frac{\partial}{\partial \nu_{\y}} G(\x-\y,t-\tau) \mu(\y,\tau) ds_{\y}d\tau\, .
\end{aligned}
\]
The singular terms $\cs_\epsilon[\sigma]$ and $\cd_\epsilon[\mu]$ 
can be treated by asymptotic methods, with $\epsilon$ chosen to be 
sufficiently small to satisfy a user-requested tolerance using  
Corollary \ref{cor:slocasym0}.
For the near singular term $\cs_{L^*}[\sigma]$, the change of variables
$e^{-u} = t-\tau$ leads to 
\begin{equation} \label{sloc2new}
\cs_{L^*}[\sigma](\x,t)=\frac{1}{4\pi}\int_{-\log{\delta}}^{-\log{\epsilon}}\int_{\Gamma(t-e^{-u})} e^{-\frac{|\x-\y|^2 e^{u}}{4}} \sigma(\y,t-e^{-u}) 
\, ds_{\y} \, du,
\end{equation}
which is a smooth integral on the interval $[-\log{\delta}, -\log{\epsilon}]$. 
Applying Gauss-Legendre quadrature with $N$ nodes, we have
\begin{equation}\label{slocnapprox}
\cs_{L^*}[\sigma](\x,t)\approx \frac{1}{4\pi} \sum_{j=1}^N \omega_j \int_{\Gamma(t-e^{-u_j})} e^{-\frac{|\x-\y|^2}{4e^{-u_j}}} \sigma(\y,t-e^{-u_j}) ds_{\y},
\end{equation}
where $\{u_j\}_{j=1}^N$ and $\{\omega_j\}_{j=1}^N$ are Legendre nodes and 
weights translated and scaled to $[-\log{\delta}, -\log{\epsilon}]$.
Similarly, for the double layer potential, we have
\be
\cd_{L^*}[\mu](\x,t)\approx\frac{1}{8\pi}\sum_{j=1}^N \omega_j\int_{\Gamma(t-e^{-u_j})} e^{-\frac{|\x-\y|^2 e^{u_j}}{4}}
\mu(\y,t-e^{-u_j}) \frac{(\x-\y)\cdot n_{\y}}{e^{-u_j}} ds_{\y}.
\label{dlocnapprox}
\ee
Detailed error analysis of the approximations \eqref{slocnapprox}
and \eqref{dlocnapprox} can be found in \cite{wang2017thesis,wang2019acom}.
The change of variables causes the quadrature nodes to cluster exponentially
toward the singular point $\tau = t$.
The net error
in the hybrid scheme for the local parts of layer potentials satisfies an
error estimate of the form:
\begin{equation}
 c_1\delta^k + c_2\sqrt{\delta} (1/N)^N),
\label{errortot}
\end{equation}
where the first term in \eqref{errortot} is the error made in
interpolating the density function at the 
Legendre nodes on the exponentially graded mesh using a $k$th order
approximation of the density. The second term comes from 
estimates of the Gauss-Legendre error on the mapped interval
$[t-\delta,t-\epsilon]$. Note that it is formally low-order accurate
in $\delta$ but one can simply set $N$ sufficiently large that
the desired precision is achieved.
The hybrid scheme is immune from geometrically induced stiffness
and has the advantage that the kernels in the spatial integrals
\eqref{slocnapprox} and \eqref{dlocnapprox}
are simply Gaussians, permitting application of the boundary FGT.

\begin{remark} \label{bootstrapeval}
The near history of layer potentials can be evaluated using the
method of this section by simply applying the quadrature component 
to the time interval
$[t-2\delta,t-\epsilon]$ rather than
$[t-\delta,t-\epsilon]$.
\end{remark}

In summary, 
initial, volume and layer potentials can all be evaluated using variants
of the FGT - simplifying the overall complexity of the code in a framwework
that can solve interior, exterior and periodic problems in optimal (or nearly
optimal) time.

\subsection{A brief review of alternative fast algorithms}

The splitting of layer heat potentials into the sum of a local part
\eqref{eq:loc} and a history part 
\eqref{eq:hist} was first introduced as a numerical tool in
\cite{greengard1990cpam} for interior problems.
Here, we outline a modification of that scheme that 
can be applied to both interior or exterior problems
\cite{lin1993thesis,greengard2000acha}, and restrict our attention
to the single layer potential. 
The double layer (or volume) potential is treated in an anaogous fashion.
The starting point of the analysis 
is the following well-known Fourier representation of the heat kernel in 
one dimension:
\begin{equation}\label{frep1d}
  G(x,t)=\frac{e^{-x^2/4t}}{\sqrt{4\pi t}}=\frac{1}{2\pi}\int_{-\infty}^{\infty}
  e^{-\xi^2t}e^{i\xi x}d\xi,
\end{equation}
from which 
\begin{equation}\label{frep2d}
  G(\x,t)=\frac{e^{-|\x|^2/4t}}{{4\pi t}}=
\frac{1}{4\pi^2}\int_{-\infty}^{\infty}\int_{-\infty}^{\infty}
  e^{-|\bxi|^2t}e^{i\bxi \cdot \x}d\bxi.
\end{equation}

\subsubsection{Spectral methods}

In \cite{greengard2000acha}, a spectral approximation
was developed for the one-dimensional heat kernel,
assuming $t\geq \delta >0$ and $|x|\leq R$, which uses
$N_F$ terms and achieves a uniform error bound of the form:
\begin{equation}
  \left|G(x,t)-\sum_{i=1}^{N_F}w_i e^{-\xi_i^2t}e^{i\xi_ix}\right|\leq \epsilon,
\quad |x|\leq R, \quad t\geq \delta.
\end{equation}
By careful selection of quadrature nodes and weights, it was shown that
the number of Fourier modes needed is of the order:
\begin{equation}
  N_F = O\left(\log\left(\frac{1}{\epsilon}\right) \left(
  \log\left(\frac{1}{\epsilon\sqrt{\delta}}\right)\right)^{1/2}
  \frac{R}{\sqrt{\delta}}
  \right).
\end{equation}
These nodes are exponentially clustered toward the origin in the Fourier
transform domain.
Spectral approximations in higher dimensions
are then easily obtained using tensor products 
\cite{greengard2000acha}, leading to
$N_F^2$ Fourier modes for problems in $\mathbb{R}^2$:
\be\label{frep2db}
G(\x-\y, t-\tau)
\approx
\sum_{j=1}^{N_F^2}
w_je^{-\|\bxi_j\|^2(t-\tau)}e^{i\bxi_j\cdot (\x-\y)}.
\ee

If we substitute
the above approximation into the history part of the
single layer potential, we obtain
\be\label{shistfrep}
\ba
\cs_H[\sigma](\x,t)=&\int_0^{t-\delta}\int_{\Gamma(\tau)}
G(\x-\y, t-\tau)
\sigma(\y(s,\tau),\tau)ds_{\y(\tau)}d\tau\\
\approx& \sum_{j=1}^{N_F^2}
\int_0^{t-\delta}\int_{\Gamma(\tau)} w_je^{-\|\bxi_j\|^2(t-\tau)}e^{i\bxi_j\cdot (\x-\y)}
\sigma(\y(s,\tau),\tau)ds_{\y(\tau)}d\tau\\
=&\sum_{j=1}^{N_F^2}w_je^{i\bxi_j\cdot \x(t)}H_j(t),
\ea
\ee
where we have interchanged the order of summation and integration.
The coefficient for each ``history mode", namely $H_{j}(t)$ 
is given by the formula
\begin{equation}\label{histmode}
H_{j}(t)=\int_0^{t-\delta}e^{-\|\bxi_j\|^2(t-\tau)}V_j(\tau)
d\tau,
\end{equation}
with $V_j$ given by the formula
\be\label{histmode2}
V_j(\tau) = \int_{\Gamma(\tau)}
e^{-i\bxi_j\cdot \y(s,\tau)}
\sigma(\y(s,\tau),\tau)ds_{\y(\tau)}, \quad j=1,\ldots,N_F^2.
\ee

As the kernels in \eqref{histmode} are simply exponential functions,
each history mode 
can be evaluated in $\mathcal{O}(1)$ work per time step
using the following recurrence relation
\cite{greengard2000acha,greengard1990cpam,lin1993thesis}
\be\label{hmoderecurrence}
H_{j}(t+\dt)= 
e^{-\|\bxi_j\|^2 \dt} H_{j}(t) +
\int_{t-\delta}^{t+\dt-\delta}
e^{-\|\bxi_j\|^2 (t+\dt-\tau)} V_j(\tau)
d\tau.
\ee
The ``update" integrals in \eqref{hmoderecurrence} can be discretized with
$k$th order accuracy by using a suitable $k$-point quadrature rule.
The spatial integrals $V_j(\tau)$ in \eqref{histmode2} can be
evaluated using the non-uniform FFT of type 3 
(e.g. \cite{finufftlib,nufft2,nufft3,nufft6,nufftlib,nufft7,nufft8})
with $O((N_S+N_F^2)\log(N_S+N_F^2))$ work per time step, where $N_S$ denotes
the number of points in the discretization of $\Gamma$.
Finally,
we may apply the type-3 NUFFT again to evaluate the last expression
 in \eqref{shistfrep} at all $N_S$ boundary points, or $N$ volumetric
discretization points,
at a cost of the order
$O((N_S+N_F^2)\log(N_S+N_F^2))$ or 
$O((N+N_F^2)\log(N+N_F^2))$, respectively.
The algorithm has a net computational complexity of the order
$O(N_T(N_S+N_F^2)\log(N_S+N_F^2))$ for evaluation on the space-time
boundary over $N_T$ time steps. The full algorithm
has been implemented in \cite{wang2016njit,wang2019jsc}. For bounded
domains, see also \cite{brattkus1992siap,ibanez2002,sethianstrain}.

\begin{remark}
Spectral methods for the heat equation have several advantages.
They are straightforward to apply, they do not 
require the extension of the spatial grid to capture diffusion into 
free-space,and they do not involve marching in the sense of 
repeatedly solving an initial value problem at each time step. 
The recurrence (apart
from the update integral) is a {\em multiplicative} factor applied to 
each Fourier mode. 
As a result,the spectral approach avoids the potential accumulation of 
error of the type discussed after Remark \ref{ndegrmk} above.
The evaluation of volume potentials using spectral approximation
is discussed in \cite{li2007jcp}.
The principal difficulty with this approach is that it works
best for {\em large} time steps, since the
number of Fourier modes required is of the order $O(\delta^{d/2})$. This
prevents optimal computational complexity 
as $\delta \rightarrow 0$, unless significant changes are made
to the scheme.
\end{remark}

\subsubsection{The sum-of-exponentials method}

An alternative global basis for representing the heat kernel in 
$\mathbb{R}^d$ is presented in~\cite{jiang2015acom}.
The construction
is based on an efficient sum-of-exponentials approximations for the 
1D heat kernel and the power function $1/t^\beta$ ($\beta>0$). 
More precisely, 
it is shown in \cite{jiang2015acom} that 
for any $0 < \epsilon < 0.1$ and any $T\geq 1000\delta>0$,
the 2D heat kernel $G(\x,t)$ admits the 
following approximation:
\begin{equation}\label{soe1}
\tilde{G}(\x,t)=\sum_{j=1}^{N_2} \tilde{w}_j e^{-\lambda_j t}
\sum_{k=-N_1}^{N_1} w_k e^{s_k t}e^{-\sqrt{s_k} \|\x\|}
\end{equation}
such that 
\begin{equation}\label{soe2}
|G(\x,t)-\tilde{G}(\x,t)|<\frac{1}{t^{d/2}}\cdot \epsilon
\end{equation}
for any $\x\in \mathbb{R}^d$, $t\in [\delta,T]$. 
Here $N_1$ is of the order
\begin{equation}\label{n1estimate}
  O\left(\log\left(\frac{T}{\delta}\right)
  \left(\log\left(\frac{1}{\epsilon}\right)
+\log\log\left(\frac{T}{\delta}\right)\right)\right),
\end{equation}
and $N_2$ is of the order
\be\label{n2estimate}
O\left(\log\left(\frac{1}{\epsilon}\right)
\left(\log\log\left(\frac{1}{\epsilon}\right)
+ \log\left(\frac{T}{\delta}\right)\right)\right).
\ee
The approximation \eqref{soe1} is valid in all of $\mathbb{R}^2$, and
the number of complex exponentials depends only
logarithmically on the ratio $T/\delta$. To give a sense of the 
constants involved in this approximation,
it is shown in~\cite{wang2016njit} 
that for $t\in [10^{-3},1]$ and $\epsilon=10^{-9}$, it is sufficient to 
let $N_1=47$ and $N_2=22$.
Using this representation for the heat kernel,
the history part of the single layer potential,
assuming the boundary is stationary, 
can be written in the form:
\begin{equation}\label{soe3}
\begin{aligned}
\cs_H[\sigma](\x,t)&\approx\int_0^{t-\delta}\int_{\Gamma} 
\sum_{j=1}^{N_2} \tilde{w}_j e^{-\lambda_j (t-\tau)}\\
&\sum_{k=-N_1}^{N_1} w_k e^{s_k (t-\tau)}e^{-\sqrt{s_k} \|\x-\y(s)\|}
\sigma(\y(s),\tau)ds d\tau\\
&=\sum_{j=1}^{N_2} \tilde{w}_j \sum_{k=-N_1}^{N_1} w_k H_{j,k}(\x,t).
\end{aligned}
\end{equation}
Here, each history mode $H_{j,k}$ is given by the formula
\begin{equation}\label{soehistmode}
H_{j,k}(\x,t)=\int_0^{t-\delta}
e^{(-\lambda_j+s_k) (t-\tau)} V_k(\x,\tau)
d\tau,
\end{equation}
with $V_k$ given by the formula
\begin{equation}\label{soevk}
V_k(\x,\tau)=\int_{\Gamma} e^{-\sqrt{s_k} \|\x-\y(s)\|}
\sigma(\y(s),\tau)ds.
\end{equation}
The spatial integrals in \eqref{soevk}
can be evaluated using a variety of fast 
algorithms 
\cite{HSS,fong2009jcp,martinsson,Hackbusch2015,ho2012sisc,ying2004jcp}.
This overall algorithm has nearly optimal complexity 
(see \cite{wang2016njit,wang2019jsc}). 

We made the assumption that the boundary is stationary, since this permits
the use of recurrence relations, 
as in \eqref{hmoderecurrence}, to compute the evolution
of each history mode.
When the domain is in motion, the locations where the layer
potential is to be evaluated are not known in advance. This requires
a more involved algorithm and additional approximations.

\subsubsection{The parabolic FMM}

In \cite{tausch2007jcp,tausch2012}, Tausch
developed the parabolic fast
multipole method (pFMM) for discrete summation involving the heat kernel. 
The pFMM
is similar to the classical FMM in its overall algorithmic structure, but
in $d+1$ dimensions, with special 
consideration for enforcing causality in the temporal variable.
The pFMM has been used in several applications
(e.g., \cite{harbrecht2011ip,harbrecht2014,messner2014jcp,messner2015sisc,tausch2009anm}).

\section{Stability analysis 
of Volterra equations} \label{sec:stabmarch}

Let us return now to the Volterra equation for the Dirichlet problem
\eqref{eq:intdir}, with $\ft(\x,t)=f(\x,t)-u^{(V)}(\x,t)$. 
Using the decomposition of the double
layer potential from the preceding section, we may write:
\be\label{marching1}
\ba
\frac{1}{2}\mu(\x,t+\delta)-\cd_L[\mu](\x,t+\delta) \hspace{2in}  \\
= -\ft(\x,t+\delta)+\cd_{NH}[\mu](\x,t+\delta)+\cd_{FH}[\mu](\x,t+\delta) \\
= -\ft(\x,t+\delta)+\cd_{NH}[\mu](\x,t+\delta)+v_{FH}(\x,t+\delta),
\ea
\ee
where we have replaced $\x_0$ by $\x$ and $t$ by $t+\delta$.
Note that the terms on the right side of \eqref{marching1}
are known so that the only unknown is $\mu(\x,\tau)$ for 
$\tau \in [t,t+\delta]$. 
Recall from Lemma \ref{initrep} that 
$v_{FH}(\x,t+\delta)$ is the initial potential
at time $t+\delta$ with data given by $v_{FH}(\x,t)+\cd_{NH}(\x,t)$ 
at time $t$ in $\bbR^2$. 

\begin{remark}
It may be interesting for some readers to note that, when
solving an interior boundary value problem with 
$\Omega(t) \subset D$, we may 
harshly truncate the domain over which the FGT is computed
to $D$ itself. The result is that the integral representation involves
a modified Green's function which sets the solution to zero outside of 
$D$, rather than the true free-space Green's function.
The difference is that the resulting density $\mu$ solved for 
in the integral equation will no longer be the same.
\end{remark}

Let us now consider perhaps the simplest marching scheme for
\eqref{marching1}, which we will refer to as the (forward) Euler method.
By this, we mean the scheme
\begin{align}
\label{biemarch}
 \mu(\x, (n+1) \dt) \; &=  \nonumber \\
 & 2 
 \int_{n\dt}^{(n+1)\dt} \int_\Gamma
\frac{\partial G(\x-\y,(n+1)\dt -\tau)}{\partial \nu_\y}
\mu(\y,n\dt)ds(\y)d\tau  \\
&  - \; 2r(\x, (n+1)\dt),  \nonumber
\end{align}
where
\[ r(\x,(n+1)\dt) = 
-\ft(\x,(n+1)\dt)+\cd_{NH}[\mu](\x,(n+1)\dt)+v_{FH}(\x,(n+1)\dt).
\]
That is, we assume $\mu(\y,t)$ is piecewise constant over each time
interval $[j\dt,(j+1)\dt]$, taking on the value $\mu(\y,j\dt)$.
This is an explicit, first order accurate formula for the value of the unknown
at the $(n+1)$st time step.

Stability analysis for fully history-dependent recursions (where there is
not a simple one-step propagator as for ordinary and partial
differential equations) is somewhat intricate and involves the consideration
of the spectra of certain Toeplitz operators. 
For the unit ball in $\mathbb{R}^d$,
it is shown in \cite{barnett2019arxiv}
that the forward Euler scheme is {\em unconditionally stable}.
More precisely, it is shown that for $T \geq 1$, 
there exists a constant $c_d$, depending on the spatial dimension $d$ 
such that
\be
\|\mu\|_2 \le c_d T^{d/2} \|\tilde{f}\|_2,
\label{diribound}
\ee
for all $N$, $\dt$ such that $N\dt \le T$.
This result holds for either the Dirichlet or Neumann problem.
In two dimensions, it is further shown that
\be
\|\mu\|_2 \le 7 \|\tilde{f}\|_2
\label{diribound2}
\ee
for any $N$, so long as $\dt\le 1$.
Since we have assumed that the 
diffusion coefficient has been scaled to one, 
this is a physically reasonable requirement on the time step.
The stability result for the Dirichlet problem is also extended in 
\cite{barnett2019arxiv} to the case of an arbitrary
$C^1$ convex boundary in any dimension.
For Robin (convective) boundary conditions, it was shown that
the Euler scheme is stable if
$\dt < \pi/(c^2 \kappa^2)$, where $c = 3 - \sqrt{2}$ and
$\kappa$ is the heat transfer coefficient,
independent of the spatial discretization. 

Such stability results are not possible for direct discretization of the 
governing partial differential equation, since the CFL 
(or domain of dependence) condition is violated for large time 
steps when there is only local coupling of the unknowns.
Numerical experiments suggest that the stability results cited above hold
for smooth, non-convex domains without significant modification,
but this remains to be proven. 

\section{Time-stepping methods and space-time adaptivity}\label{sec:alg}

Let us first consider the semilinear initial value problem
\begin{equation} \label{heatfreeam}
\begin{split}
u_{t}(\x,t) &= \Delta u(\x,t)+F(u,\x,t), \\
u(\x,0) &= u_0(\x),
\end{split}
\end{equation}
with periodic boundary conditions imposed on the unit box 
$B=[-0.5,0.5]^2$. 
As noted in the introduction, we may represent the 
solution in terms of an initial and volume potential:
\be
u(\x,t)=\I[u_0](\x,t)+\V[F](\x,t)
\label{freesolam}
\ee
with 
\begin{align*}
\I[u_0](\x,t) &= \int_{B} G_p(\x-\y,t) u_0(\y) \, d\y,  \\
\V[F](\x,t) &= \int_0^t\int_{B} G_p(\x-\y,t-\tau) F(u,\y,\tau) \, d\y d\tau,
\end{align*}
where $G_p(\x,t)$
is the periodic Green's function for the heat equation 
on the unit box \eqref{heatker2dper}.

\subsection{The Adams-Moulton method for the semilinear Volterra equation}

Let us denote the $n$th time step by 
$t_n=n\dt$ and the solution at the $n$th time step by $u_n(\x)=u(\x,n\dt)$
and consider the marching scheme proposed in
\eqref{vmarchnonlin} written in the form
\be
u_{n+1}(\x)=\I[u_{n}](\x,\dt)+\int_{n\dt}^{(n+1)\dt}U(u,\x,\tau)d\tau
\label{am1}
\ee
with
\be
U(u,\x,\tau)=\int_{B} G_p(\x-\y,t_{n+1}-\tau) F(u(\y,\tau),\y,\tau) \, d\y.
\label{am2}
\ee

Following the nomenclature from ordinary differential equations (ODEs),
we refer to the scheme \eqref{am1} as an
Adams-Moulton method of order $s$, with $s \geq 1$, 
if we approximate $U(\x,\tau)$ for each $\x$ by an interpolating 
polynomial in time, $p(\x,\tau)$, of degree $s-1$:
\be
p(\x,t_{n+1-i})=U(u_{n+1-i},\x,t_{n+1-i}), \qquad i=0,\ldots,s-1.
\label{am3}
\ee
Replacing $U$ in \eqref{am1} by the interpolant $p$ and integrating
the resulting expression, we obtain
\be
u_{n+1}(\x)=\I[u_{n}](\x,\dt)+\dt \sum_{i=0}^{s-1} 
b_iU(u_{n+1-i},\x,t_{n+1-i}),
\label{am4}
\ee
where the $\{ b_i \}$ are the coefficients of the classical 
Adams-Moulton method for ODEs.
The first several Adams-Moulton schemes are as follows:
\be
\ba
&(s=0) \rightarrow b_0 = 1,\quad \\
&(s=1) \rightarrow b_0 = \frac12, b_1 = \frac12,\\
&(s=2) \rightarrow b_0 = \frac5{12}, b_1 = \frac23, b_2 = -\frac1{12},\\
&(s=3) \rightarrow 
b_0 = \frac9{24}, 
b_1 = \frac{19}{24}, b_2 = -\frac5{24}, b_3 = \frac1{24}.
\ea
\ee
Note that at time $t_{n+1}$, 
the previous solutions $u_{n+1-i}$ ($i>0$) are already
known (and can be stored). Thus, the initial potential 
$\I[u_{n}](\x,\dt)$ and the spatial voilume integrals
$U(u_{n+1-i},\x,t_{n+1-i})$ ($i>0$) are known explicitly and can be
computed in optimal time with the FGT.
Second, using the delta function property of the Green's function,
we have
\be
\ba
U(u_{n+1},\x,t_{n+1})&=\lim_{\tau\rightarrow t_{n+1}}\int_{B} G_p(\x-\y,t_{n+1}-\tau) F(u(\y,\tau),\y,\tau) \, d\y\\
&=F(u_{n+1},\x,t_{n+1}).
\label{am6}
\ea
\ee
Combining these two observations, any Adams-Moulton
marching scheme of the form \eqref{am4}
leads to a set of {\it spatially uncoupled, scalar}
nonlinear equations to obtain the new value $u_{n+1}(\x)$:
\be
u_{n+1}(\x)=g(\x)+\dt b_0 F(u_{n+1},\x,t_{n+1}),
\label{am7}
\ee
for some known function $g(\x)$.
The spatial correlation of the solution is propagated from previous time
steps via the non-local Gauss transforms applied to previous solutions 
and forcing terms. As noted in 
\cite{epperson2}, this makes solving 
\eqref{am4} much easier than applying
a fully implicit scheme to the partial differential equuation itself. 

\begin{remark}
In the PDE framework, solving only scalar nonlinear equations can be achieved
by the use of {\em operator splitting methods}, but the order of accuracy
is then limited by the splitting error, although remedies for this are available
\cite{hansen,jahnke,mclachlan,murua,strang,tyson}.
\end{remark}

\begin{remark}
  Many root finding methods can be used to solve \eqref{am7}. We use the
  secant method with two initial guesses chosen to be $u^{(0)}_{n+1}(\x)=u_{n}(\x)$
  and $u^{(1)}_{n+1}(\x)=g(\x)+\dt\, b_0\, F(u_{n},\x,t_{n})$. 
\end{remark}
\begin{remark}
  When the order $s$ of the Adams-Moulton scheme is greater than $2$, 
one needs $u_1, \ldots, u_{s-1}$ with high order accuracy to initialize
the marching scheme. This can be accomplished using a second order 
Adams-Moulton scheme and Richardson extrapolation to the desired order.
\end{remark}
We summarize the steps for solving \eqref{heatfreeam} in Algorithm~\ref{amscheme}.
\begin{algorithm} 
\caption{The Adams-Moulton scheme for solving \eqref{heatfree}}
\label{amscheme}
\begin{algorithmic}
  \Procedure{h2dpersemilinearam}{$\dt$, $N_T$, $s$, $u$, $F$}
  \LeftComment{Input arguments: $\dt$ - the time step.\\
    $\qquad\qquad\qquad\qquad\qquad\,\,\,\,\, N_T$ - the total number of time steps.\\
    $\qquad\qquad\qquad\qquad\qquad\,\,\,\,\, s$ - the order of the Adams-Moulton scheme.\\
    $\qquad\qquad\qquad\qquad\qquad\,\,\,\,\, u$ - the initial data.\\
    $\qquad\qquad\qquad\qquad\qquad\,\,\,\,\, F$ - the function $F(u,\x,t)$.}
  \LeftComment{Output argument: $u$ - the solution to \eqref{heatfree} at time $N_T\dt$.}
  \If {$s > 2$}
  \State Use the second order Adams-Moulton scheme and Richardson extrapolation
   to compute $u_1, u_2,...,u_{s-1}$.
  \EndIf
  \For {$n = s-1, N_T-1$}
  \State Use the FGT to compute $\I[u_{n}](\x,\dt)$ and
  $U(u_{n+1-i},\x,t_{n+1-i})$ ($1\le i<s-1$).
  \State Use the secant method to solve \eqref{am7} to obtain $u_{n+1}(\x)$.
  \EndFor
  \EndProcedure
\end{algorithmic}
\end{algorithm}

\subsection{Spatial refinement and coarsening at each time step} \label{coarseningsec}
For the linear case, the quad-tree structure for each time $t_n$ can be build
in a straightforward manner, as discussed in section \ref{sec:funcapprox}.
That is, for any box $B$, we can form the Chebyshev interpolant of the function
values $F(\x,t)$ from its values
at the $K\times K$ Chebyshev grid in the box. 
We can then evaluate $F(\x,t)$ on the $2k \times 2k$ grid corresponding to the
Chebyshev grids of $B$'s children and see if it agrees with the interpolant to the 
desired precision. 
Alternatively, we could simply develop a heuristic based on the
decay of the Chebyshev expansion computed on $B$.
If the agreement does not satisfy the error tolerance (or the decay heuristic
is not satisfied), we subdivide the box and repeat for each child.

For the semilinear case, the value of the function $F(u,\x,t)$ is not easily
obtained at an arbitrary point $\x$ at $t_{n+1}$ since an accurate value of $u$ is not 
readily available.
However, since \eqref{am7} is a scalar nonlinear equation for each $\x$, 
we may solve the equation at the new point to obtain $u(\x,t_{n+1})$ and
$F(u,\x,t_{n+1})$, after which the same adaptive refinement strategy can be employed.
This permits a local spatial refinement strategy (which is easily executed in parallel).
Note that the function 
$g(\x)$ in \eqref{am7} can be assumed to be resolved as it corresponds to previously
computed solutions. Thus, it can be obtained at the child grid nodes
by Chebyshev interpolation from $B$. 

For coarsening, the reverse strategy is employed. Suppose $B'$ 
is a box with four children.
We can compute $u(\x,t_{n+1})$ at the $K \times K$ Chebyshev nodes in $B'$.
If the corresponding interpolant from $B$ matches both $u(\x,t_{n+1})$ and 
$F(u,\x,t_{n+1})$ to the desired precision at all child grid points, then we 
delete the child nodes and retain the coarsened grid on $B'$ for subsequent time steps.
This is summarized in Algorithm~\ref{spatialadap}.
\begin{algorithm} 
\caption{Spatial refinement and coarsening algorithm for the semilinear heat equation}
\label{spatialadap}
\begin{algorithmic}
  \Procedure{spatialadap}{$T_{n}$, $F$, $n_{\rm l}$, $\epsilon$, $T_{n+1}$, $u_n$}
  \LeftComment{Input: $T_{n}$ - the adaptive quad-tree at $t_{n}$.\\
    $\qquad\qquad\qquad\, F$ - the function $F(u,\x,t)$.\\
    $\qquad\qquad\qquad\, n_{\rm l}$ - the number of levels of $T_{n}$.\\
    $\qquad\qquad\qquad\,\,\, \epsilon$ - the desired precision.}
  \LeftComment{Output: $T_{n+1}$ - the adaptive quad-tree at $t_{n+1}$.\\
    $\qquad\qquad\qquad\,\,\,\, u_{n+1}$ - the solution $u_{n+1}$ 
on $T_{n+1}$.}
        
  \LeftComment{{\bf Step 1} Refinement:}
  \For {each box $B$ in $T_{n}$}
  \If {$B$ is childless}
  \State Calculate the Chebyshev interpolant of $F(u_{n+1},\x,t_{n+1})$ on $B$ and
  compute $F(u_{n+1},\x,t_{n+1})$ on a refined grid corresponding to $B$'s children.
  Compute the error $E$ in the interpolant on the child grids.
  \If {$E < \epsilon$}
  \State \textbf{break}
  \Else
  \State Add the four child boxes to the quad-tree data structure.
  \EndIf
  \EndIf
  \EndFor
  
  \LeftComment{{\bf Step 2} Coarsening:}
  \For {$i = n_{\rm l}-1, 0, -1$}
  \State Calculate the values of the function $F(u_{n+1},\x,t_{n+1})$ 
on the $K\times K$ grid points for each box that has 
  four child leaf nodes.
  \State Delete the four child boxes if $u_{n+1}$ and $F(u_{n+1},\x,t_{n+1})$ are resolved
  at the parent level.
  \EndFor
  
  \EndProcedure
\end{algorithmic}
\end{algorithm}

\section{Numerical examples}\label{sec:examples}

In this section, we illustrate the performance
of the methods described above for solving a variety of problems governed by
the heat equation. All of the component algorithms were implemented in Fortran
and all timings listed are generated using a laptop with a 2.7GHz Intel Core i5 processor and 8GB RAM.

\subsection{A linear, heat equation in a periodic box}
For our first example, we consider
the inhomogeneous heat equation \eqref{heatfree} subject to periodic boundary
conditions on the box $B=[-0.5, 0.5]^2$, as discussed in section \ref{periodicbc},
with the forcing function $F(\x, t)$ given by 
\be \label{eq:ex1}
F(\x, t) = \sum_{\j\in\Z^2} e^{-|\x-\mathbf{c}_1(t)-\j|^2/\delta}-0.5 \, e^{-|\x-\mathbf{c}_2(t)-\j|^2/\delta},
\ee
where $\mathbf{c}_1(t)=(0.25 \, \cos(20\pi t), 0.25 \, \sin(20\pi t))$, and 
$\mathbf{c}_2(t)=(0.25 \, \cos(40\pi t+\pi), 0.25 \, \sin(40\pi t+\pi))$.
These two points move on a circle at different velocities and cause $F(\x,t)$ to
be highly inhomogeneous in space-time.
Note that the evaluation of the forcing function requires summation over the two-dimensional
lattice of images. In practice, the nearest twenty-five copies 
($|\j|\leq 2$) is sufficient to guarantee periodicity with 12 digits of accuracy.
The time-varying adaptive spatial grid
used for resolving the solution is shown in Figure \ref{fig:adaptive_forcing}. It is
obtained automatically using the method of section \ref{coarseningsec}.

\begin{figure}[H]  
  \includegraphics[width=\textwidth]{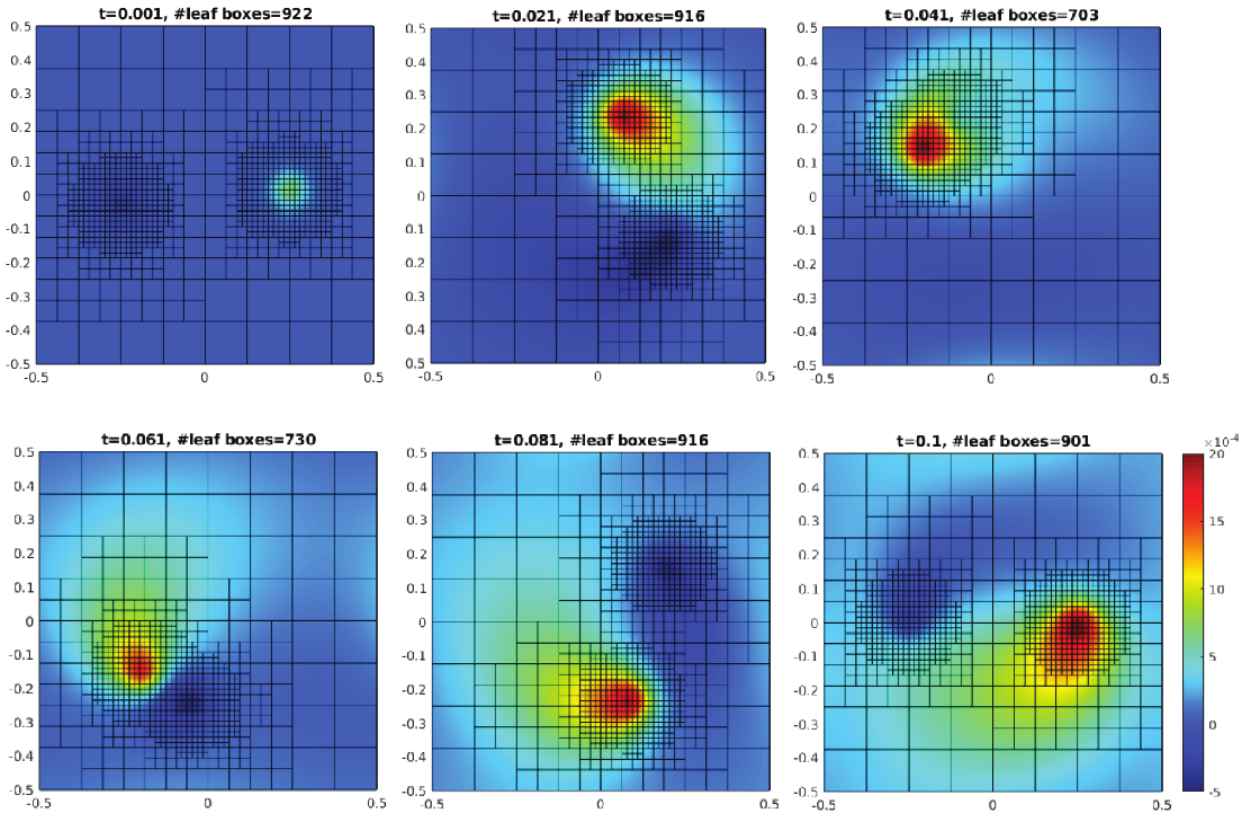}
\caption{Evolution of the solution to the inhomogenous, periodic heat equation with forcing
  function defined in \eqref{eq:ex1}. Also shown are the level-restricted quadtrees at the 
corresponding times generated using our automatic refinement/coarsening strategy
with an eighth-order, piecewise Chebyshev polynomial approximation and an error 
tolerance of $\epsilon^* = 10^{-8}$. }
\label{fig:adaptive_forcing}
\end{figure}

\begin{table}[!ht]
  \caption{Performance of our solver when applied to the inhomogenous heat equation 
with forcing function given in \eqref{eq:ex1}. Here, $|\ell|_{max}$ is the maximum 
number of leaf nodes over the quadtrees at all time steps and $N_T$ is the number of 
time steps using a fourth order accurate quadrature for the local part
(see section \ref{sec:loc}). The reported error is measured in the $L^2$-norm at the final time step.}
\centering
\begin{tabular}{@{}lcccc@{}}
  \toprule
  $\epsilon^*$  & $10^{-3}$ & $10^{-6}$ & $10^{-9}$\\
  \midrule 
   $|\ell|_{max}$ & 73 & 256 & 952  \\ 
   $N_T$ & 25 & 200 & 1000  \\ 
   error & $1.5\cdot 10^{-2}$ & $4.0 \cdot 10^{-5}$ & $8.3 \cdot 10^{-8}$ \\
 \bottomrule
 \end{tabular}
\label{table_errors}
\end{table}

Note that the order
of accuracy in space and time is sufficiently high that 
the error is dominated by the choice of $\epsilon^*$, which 
determines the accuracy of the representation of the forcing function.

\subsection{A semilinear heat equation in a periodic box}

For our second example, we again impose periodic boundary conditions on the box
$B=[-0.5, 0.5]^2$, but consider a semilinear heat equation using the implicit
Adams-Moulton integrator. Recall that, in the integral equation formulation,
this requires only the solution of scalar nonlinear equations.
We solve
\be
\ba
u_t(\x,t)&=\Delta u(\x,t)+F(u,\x,t), \qquad \x\in B\\
u(\x,0)&=f(\x,0), \qquad \x\in B,
\ea
\label{semiex1}
\ee
where $F(u,\x,t)=u^2(\x,t)-f^2(\x,t)+f_t(\x,t)-\Delta f(\x,t)$, with exact solution
\be
f(\x,t)=\sin(2\pi x_1)\cos(2\pi x_2)e^{-t} + \cos(2\pi x_1)\sin(2\pi x_2)e^{-3t}.
\ee
Table \ref{table1} shows the relative $L^2$ errors of the numerical solution
with respect to the exact solution at $T=0.2$ for various time steps and
orders of accuracy. Table \ref{table2} then lists the ratios of successive errors,
which are approaching the theoretical value $2^p$ for the $p$th
order scheme as $\dt$ decreases.

\begin{table}[!ht]
\caption{The relative $L^2$ error of the Adams-Moulton scheme for 
solving \eqref{semiex1}. The column heading $A_p$ indicates the use of
the $p$th order Adams-Moulton scheme.}
\centering
\begin{tabular}{rllllll}
  \toprule
  ${N_T}$  & ${\dt}$ & ${A_2}$ & ${A_3}$ & ${A_4}$ & ${A_5}$ & ${A_6}$  \\
  \midrule
   10 & ${1}/{50}$ & $2.3\cdot 10^{-1}$ & $1.1 \cdot 10^{-1}$ & $6.1\cdot 10^{-2}$ & $3.8\cdot 10^{-2}$ & $2.6\cdot 10^{-2}$ \\
   20 & ${1}/{100}$ & $5.9\cdot 10^{-2}$ & $1.7\cdot 10^{-2}$ & $6.3\cdot 10^{-3}$ & $2.6\cdot 10^{-3}$ & $1.2\cdot 10^{-3}$ \\
   40 & ${1}/{200}$ & $1.5\cdot 10^{-2}$ & $2.57\cdot 10^{-3}$ & $5.3\cdot 10^{-4}$ & $1.3\cdot 10^{-4}$ & $3.5\cdot 10^{-5}$ \\
   80 & ${1}/{400}$  & $3.7\cdot 10^{-3}$ & $3.4\cdot 10^{-4}$ & $3.8\cdot 10^{-5}$ & $4.9\cdot 10^{-6}$ & $8.0\cdot 10^{-7}$ \\
  160 & ${1}/{800}$  & $9.3\cdot 10^{-4}$ & $4.4\cdot 10^{-5}$ & $2.6\cdot 10^{-6}$ & $1.7\cdot 10^{-7}$ & $1.7\cdot 10^{-8}$ \\
  320 & ${1}/{1600}$  & $2.3\cdot 10^{-4}$ & $5.6\cdot 10^{-6}$ & $1.7\cdot 10^{-7}$ & $5.7\cdot 10^{-9}$ & $2.6\cdot 10^{-10}$ \\
 \bottomrule
 \end{tabular}
\label{table1}
\end{table}

\begin{table}[!ht]
  \caption{Ratios of succesive errors in Table \ref{table1}.
  }
  \sisetup{
  table-format = 3.1,
  table-auto-round=true,
  tight-spacing=true
}
\centering
\begin{tabular}{S[table-format=1.1]S[table-format=1.1]SSS}
  \toprule
  ${A_2}$ & ${A_3}$ & ${A_4}$ & ${A_5}$ & ${A_6}$  \\
  \midrule
     3.94 &   6.21 &   9.62 &  14.79 &  22.15 \\
     3.98 &   6.97 &  12.04 &  20.67 &  34.09 \\
     4.00 &   7.45 &  13.78 &  25.39 &  43.59 \\
     4.00 &   7.71 &  14.82 &  28.41 &  48.06 \\
     4.00 &   7.85 &  15.40 &  30.54 &  63.63 \\
 \bottomrule
 \end{tabular}
\label{table2}
\end{table}

\subsection{The Fujita Model}

We now consider a classical reaction-diffusion process described
by the equation
\be
u_t=\Delta u +u^p.
\label{fujita1}
\ee
When $p>1$, it has been shown in \cite{fujita1966} that the solution
may blow up in finite time for the pure initial problem.
For zero Dirichlet boundary conditions, with intial data given
by a Gaussian, the solution will blow up in in finite time if the
initial data is large (see, for example, \cite{vazquez2017} and the
references therein). Here, we study the behavior of the solution
with periodic boundary conditions. We choose $p=2$ and the initial data
as a periodized Gaussian 
\be
u_0(\x)=5\sum_{\j\in\Z^2} e^{-\frac{\|\x- \j\|^2}{0.1}}, \qquad \x\in B.
\label{fujita2}
\ee
As above, we truncate the infinite lattice in the preceding expression at
($|\j|\leq 4$), which is sufficient for full double precision accuracy.
We use the 4th order Adams-Moulton scheme as the solver and terminate
the simulation when the maximum of the solution is more than 10 times
the maximum value of the initial data. The precision flag for the FGT
is set to be $10^{-7}$. We run the simulation twice, with $\dt = 
2\times 10^{-3}$ and $\dt = 10^{-3}$.
The comparison of these two numerical solutions suggests that we have 
at least two digits of accuracy. Figure~\ref{fig:fujita1} shows
snapshots of the solution at various times.

\begin{figure}[t]  
\centering
  \includegraphics[height=45mm]{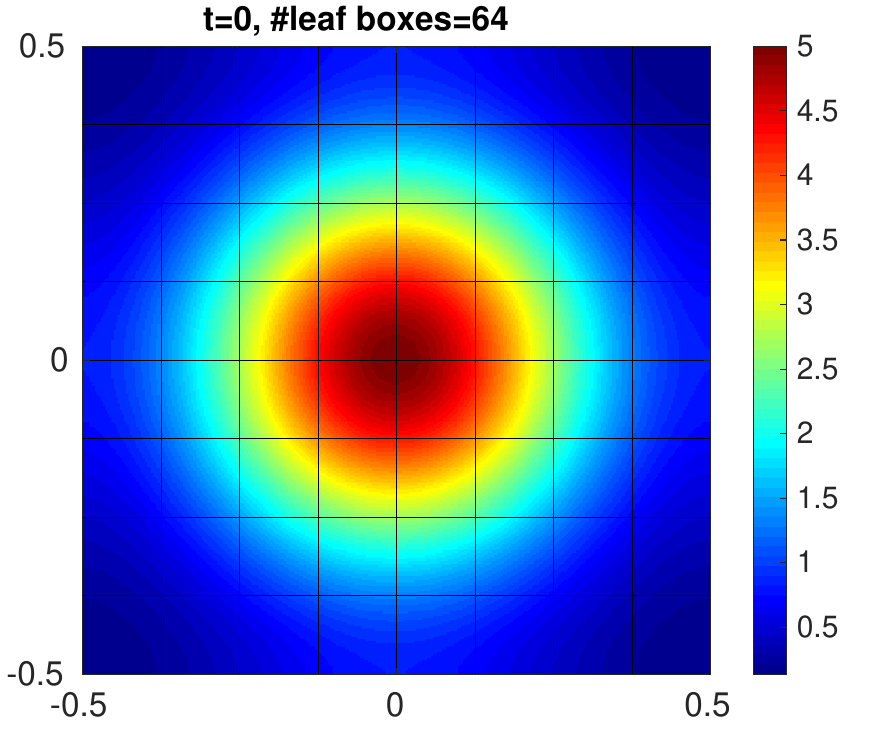}
  \includegraphics[height=45mm]{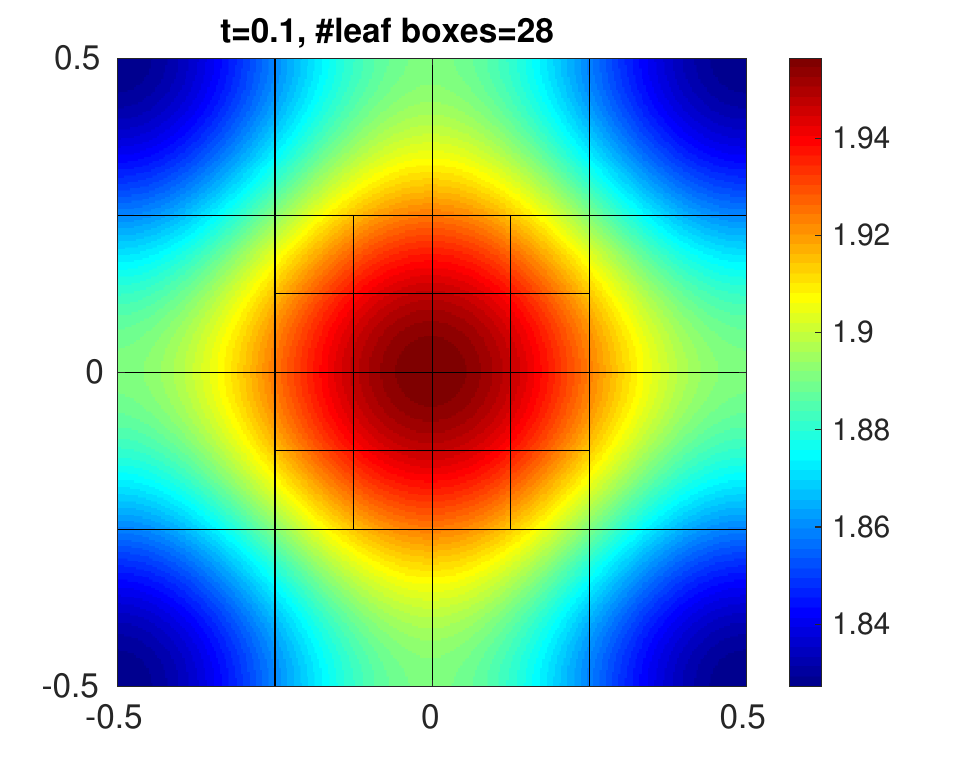}

  \vspace{5mm}

  \includegraphics[height=45mm]{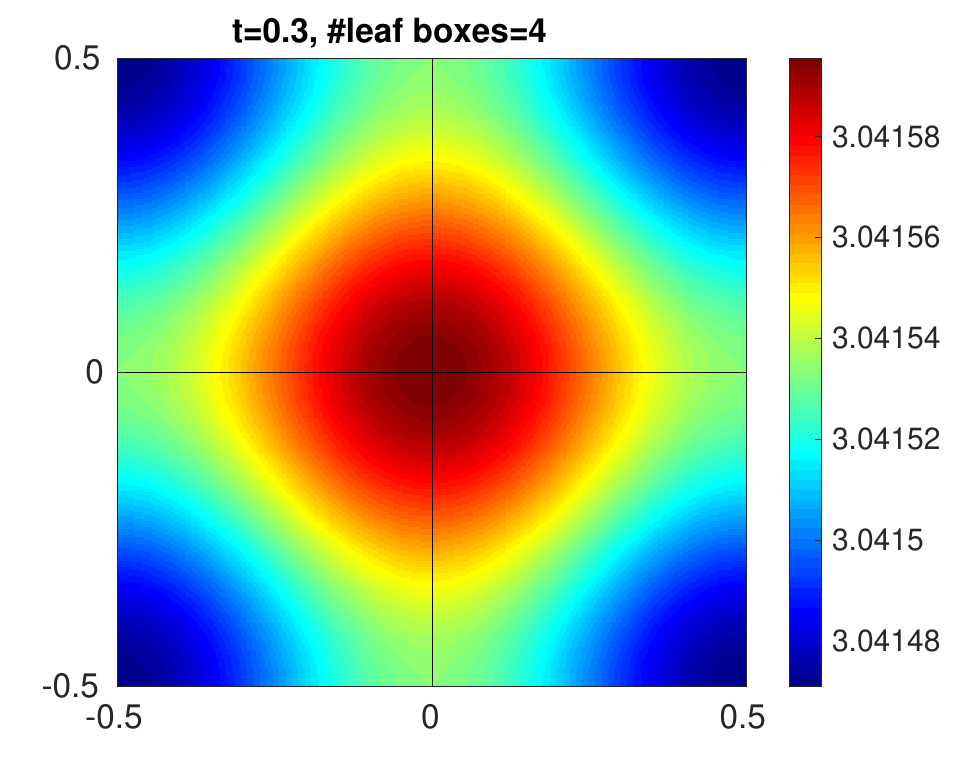}
  \includegraphics[height=45mm]{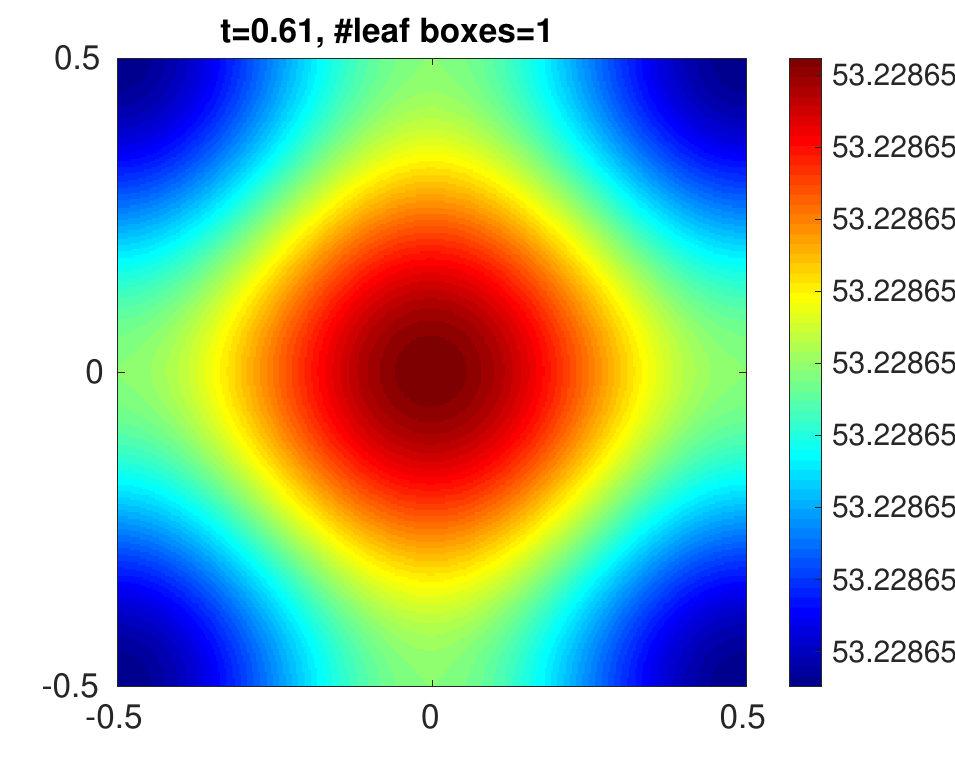}
\caption{Evolution of the solution of the Fujita model with $p=2$ and the initial
  data given by \eqref{fujita2} under periodic boundary conditions.
\label{fig:fujita1}}
\end{figure}

Figure~\ref{fig:fujita2}
shows the evolution of the maximum value of the solution
$u_{\rm max}(t)=\max_{\x\in B} u(\x,t)$. We observe that the solution
decays initially as the diffusion effect dominates the reaction component.
At about $t_c\approx 0.099$, the solution
flattens out to being nearly a constant $u_c\approx 1.91$. 
After that, the solution
evolves as nearly constant on the whole spatial domain and follows
the solution to the ODE $u_t=u^2$, given by the formula
\[\frac{1}{1/u_c-(t-t_c)}.\]
This model fits the data in Figure~\ref{fig:fujita2} for $t>t_c$ extremely well.
We conjecture that, for periodic boundary conditions, the solution will
blow up in finite time if the initial data is a periodized Gaussian, regardless
of its magnitude. This is in stark contrast to the case of zero
Dirichlet conditions. It is, perhaps, to be expected, since
Dirichlet boundary conditions can be interpreted as tiling the plane with
reflected Gaussians of alternating sign, while periodic conditions tile the plane
with images of the same sign.
We leave a rigorous analysis to the interested reader.
\begin{figure}[t]  
\centering
  \includegraphics[height=45mm]{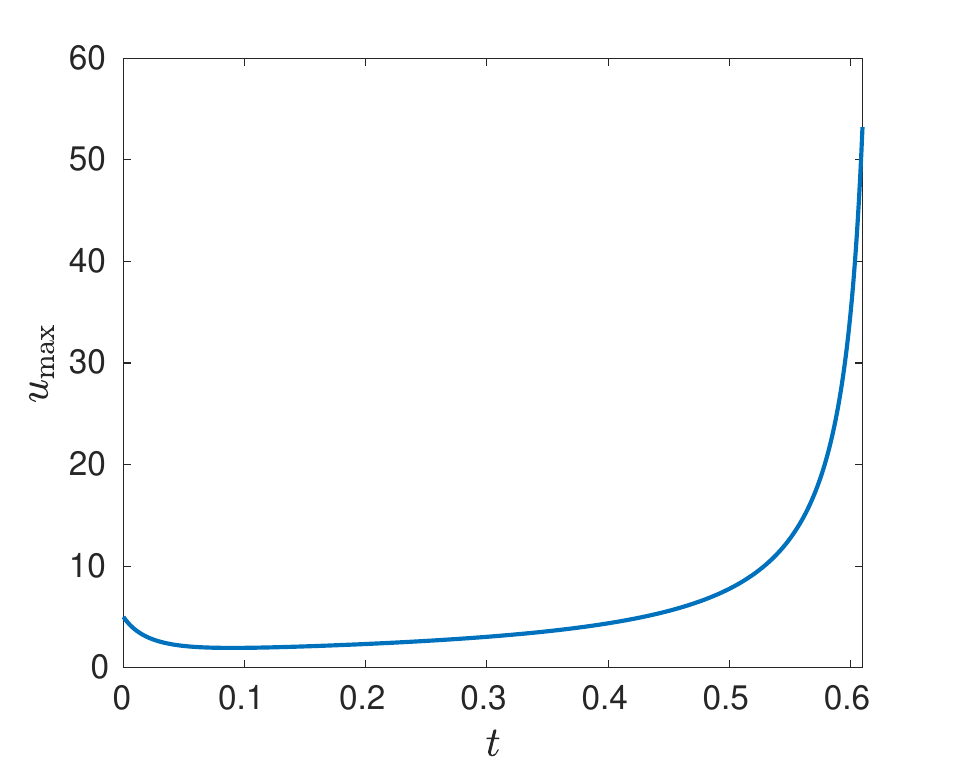}
\caption{Evolution of $u_{\rm max}(t)$.}
\label{fig:fujita2}
\end{figure}

\subsection{Heat equation in static, multiply-connected domain} \label{mcsec}
To illustrate the geometric flexibility and speed of our layer potential
evlauation code, in our last example we solve the heat equation in a periodic box containing 
73 closed curves, on which we impose the initial condition, $\ub(\x,0)=\sin(2\pi x_1)\cos(2\pi x_2)$,
and the boundary condition 
\[ \ub(\x,t)= \sin(2\pi x_1)\cos(2\pi x_2) e^{-8\pi^2 t}, \]
where $\x \in \Gamma$ (see Fig. \ref{fig:complexgeom}.)

The boundary is discretized by piecewise $16$th order Legendre polynomials, with 74752 boundary points in total. A level-restricted quadtree is created so that the initial
condition is resolved to precision $\epsilon=10^{-9}$ and that each leaf box contains
no more than $16$ points. The tree has $9$ levels and $17647$ leaf nodes with an $8 \times 8$ tensor
product Chebyshev grid in each. This leads to around $10^6$ grid points in the unit box and $\Delta x \approx 10^{-4}$ (side length of the smallest leaf node). 

We solve this boundary value problem with a second-order predictor-corrector method in time, with $8$ quadrature nodes in time each for the evaluation of $D_{NH}$ and $D_{L}$. The step size 
$\Delta t = 10^{-8}$ is chosen so that $\Delta t\approx\Delta x^2$. 
The total CPU time consumed
in one time step is $17.6$ seconds: $0.6$ seconds for $D_{FH}$, $10.8$ seconds for $D_{NH}$,
and $6.27$ seconds for $D_{L}$.  The relative $L^2$ error 
at $T=1000\Delta t$ is $2.7 \cdot 10^{-6}$. 

\begin{figure}[t]  
\centering
  \includegraphics[width=0.7\textwidth]{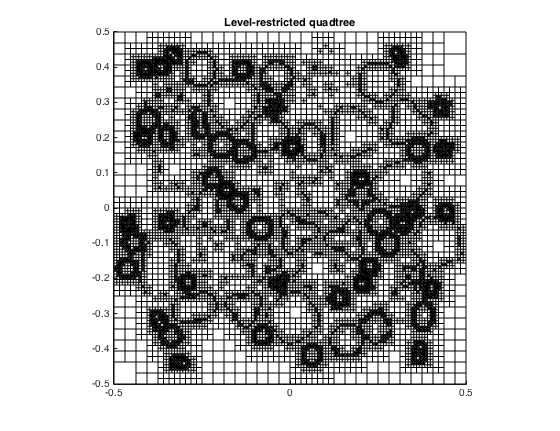}
\caption{The complex domain in our last example and
the level restricted quadtree used for evaluating the corresponding heat potentials.}
\label{fig:complexgeom}
\end{figure}

\section{Conclusions}\label{sec:conclusion}

We have presented a collection of tools for the solution of diffusion
problems in interior, exterior or periodic domains
based on initial, volume and layer heat potentials.
In the linear setting,
these tools permit the creation of stable, explicit, fully automatic,
high-order, space-time adaptive methods for complicated, moving
geometries. 
In the semi-linear setting, a fully implicit formulation needed 
to handle stiff forcing terms requires only
the solution of uncoupled scalar nonlinear equations.

The principal numerical algorithm used here is the fast 
Gauss transform, which we have shown can be used to develop
``optimal complexity"  algorithms for all types 
of heat potentials. Since there are a number of other 
fast algorithms available, future work will involve 
determining which is most suitable for implementation in
three dimensions and for high-performance computing platforms.
Open problems include the design of efficient quadratures for
discontinuous forcing functions 
(without excessive adaptivity at the discontinuities), 
the design of high order representations for 
moving geometries involving free boundaries, and the design
of robust integral representations for systems of diffusion equations.

In the present work, we have assumed that the diffusion coefficient
itself is constant. We are currently working on
integral equation methods that can handle piecewise-constant
and piecewise-smooth diffusion coefficients, and in incorporating
the solvers presented here into codes for applications including
reaction-diffusion equations, crystal growth, 
diffusion tensor magnetic resonance imaging, and viscous fluid dynamics.
Code libraries for the schemes presented here are being prepared for
open source release.

\section*{Acknowledgements}
We would like to thank Travis Askham, Alex Barnett, and Manas Rachh for many
helpful discussions.

\bibliographystyle{plain}

\begin{thebibliography}{10}

\bibitem{alpert1999sisc}
B.~K. Alpert.
\newblock Hybrid {G}auss-trapezoidal quadrature rules.
\newblock {\em SIAM J. Sci. Comput.}, 20:1551--1584, 1999.

\bibitem{arnold1989jcm}
D.~N. Arnold and P.~Noon.
\newblock Coercivity of the single layer heat potential.
\newblock {\em J. Comput. Math}, 7:100--104, 1989.

\bibitem{AskhamCerfon}
Travis Askham and Antoine~J. Cerfon.
\newblock An adaptive fast multipole accelerated poisson solver for complex
  geometries.
\newblock {\em J. Comput. Phys.}, 344:1--22, 2017.

\bibitem{finufftlib}
Alex Barnett and J.~Magland.
\newblock Non-uniform fast {F}ourier transform library of types $1$, $2$, $3$
  in dimensions $1$, $2$, $3$.
\newblock \url{https://github.com/ahbarnett/finufft}, 2018.

\bibitem{barnett2019arxiv}
Alex~H Barnett, Charles~L Epstein, Leslie Greengard, Shidong Jiang, and Jun
  Wang.
\newblock Explicit unconditionally stable methods for the heat equation via
  potential theory.
\newblock {\em arXiv preprint arXiv:1902.08690}, 2019.

\bibitem{boyd2001}
J.~P. Boyd.
\newblock {\em Chebyshev and {F}ourier Spectral Methods}.
\newblock Dover, New York, 2001.

\bibitem{brattkus1992siap}
K.~Brattkus and D.~I. Meiron.
\newblock Numerical simulations of unsteady crystal growth.
\newblock {\em SIAM J. Appl. Math.}, 52(5):1303--1320, 1992.

\bibitem{HSS}
S.~Chandrasekaran, P.~Dewilde, M.~Gu, W.~Lyons, and T.~Pals.
\newblock A fast solver for {HSS} representations via sparse matrices.
\newblock {\em SIAM J. Matrix Anal. Appl.}, 29:67--81, 2006.

\bibitem{cheng1999jcp}
H.~Cheng, L.~Greengard, and V.~Rokhlin.
\newblock A fast adaptive multipole algorithm in three dimensions.
\newblock {\em J. Comput. Phys.}, 155(2):468--498, 1999.

\bibitem{cheng2006jcp}
H.~Cheng, J.~Huang, and T.~J. Leiterman.
\newblock An adaptive fast solver for the modified {H}elmholtz equation in two
  dimensions.
\newblock {\em J. Comput. Phys.}, 211:616--637, 2006.

\bibitem{costabel}
M.~Costabel.
\newblock Time-dependent problems with the boundary integral equation.
\newblock In E.~Stein, R.~de~Borst, and T.~J.~R. Hughes, editors, {\em
  Encyclopedia of Computational Mechanics}, pages 703--721. John Wiley \& Sons,
  New York, 2004.

\bibitem{DR}
P.~J. Davis and P.~Rabinowitz.
\newblock {\em Methods of numerical integration}.
\newblock Academic Press, San Diego, 1984.

\bibitem{quadtree}
Mark de~Berg, Otfried Cheong, Markvan Kreveld, and Mark Overmars.
\newblock {\em Computational Geometry: Algorithms and Applications}.
\newblock Springer, Berlin, 2008.

\bibitem{nufft2}
A.~Dutt and V.~Rokhlin.
\newblock Fast {F}ourier transforms for nonequispaced data.
\newblock {\em SIMA J. Sci. Comput.}, 14:1368--1393, 1993.

\bibitem{nufft3}
A.~Dutt and V.~Rokhlin.
\newblock Fast {F}ourier transforms for nonequispaced data. {II}.
\newblock {\em Appl. Comput. Harmon. Anal.}, 2:85--100, 1995.

\bibitem{epperson2}
James~F. Epperson.
\newblock On the use of {G}reen's functions for approximating nonlinear
  parabolic pdes.
\newblock {\em Applied Math. Letters}, 2:293--296, 1989.

\bibitem{Ethridge2001sisc}
F.~Ethridge and L.~Greengard.
\newblock {A New Fast-Multipole Accelerated {P}oisson Solver in Two
  Dimensions}.
\newblock {\em SIAM Journal on Scientific Computing}, 23:741--760, 2001.

\bibitem{fong2009jcp}
W.~Fong and E.~Darve.
\newblock The black-box fast multipole method.
\newblock {\em J. Comput. Phys.}, 228(23):8712--8725, 2009.

\bibitem{friedman1964}
A.~Friedman.
\newblock {\em Partial Differential Equations of Parabolic Type}.
\newblock Prentice-Hall, Englewood Cliffs, New Jersey, 1964.

\bibitem{fujita1966}
Hiroshi Fujita.
\newblock On the blowing up of solutions of the {C}auchy problem for
  {$u_{t}=\Delta u+u^{1+\alpha }$}.
\newblock {\em J. Fac. Sci. Univ. Tokyo Sect. I}, 13:109--124 (1966), 1966.

\bibitem{martinsson}
A.~Gillman, P.~M. Young, and P.-G. Martinsson.
\newblock A direct solver with {$O(N)$} complexity for integral equations on
  one-dimensional domains.
\newblock {\em Front. Math. China}, 7(2):217--247, 2012.

\bibitem{gimbutas2003sisc}
Z.~Gimbutas and V.~Rokhlin.
\newblock A generalized fast multipole method for nonoscillatory kernels.
\newblock {\em SIAM J. Sci. Comput.}, 24:796--817, 2003.

\bibitem{nufft6}
L.~Greengard and J.Y. Lee.
\newblock Accelerating the nonuniform fast {F}ourier transform.
\newblock {\em SIAM Rev.}, 46:443--454, 2004.

\bibitem{greengard2000acha}
L.~Greengard and P.~Lin.
\newblock Spectral approximation of the free-space heat kernel.
\newblock {\em Appl. Comput. Harmon. Anal.}, 9:83--97, 2000.

\bibitem{greengard1987jcp}
L.~Greengard and V.~Rokhlin.
\newblock A fast algorithm for particle simulations.
\newblock {\em J. Comp. Phys.}, 73(2):325--348, 1987.

\bibitem{greengard1997actanum}
L.~Greengard and V.~Rokhlin.
\newblock A new version of the fast multipole method for the {L}aplace equation
  in three dimensions.
\newblock {\em Acta. Numer.}, 6:229--270, 1997.

\bibitem{greengard1990cpam}
L.~Greengard and J.~Strain.
\newblock A fast algorithm for the evaluation of heat potentials.
\newblock {\em Comm. on Pure and Appl. Math}, 43:949--963, 1990.

\bibitem{greengard1991fgt}
L.~Greengard and J.~Strain.
\newblock The fast {G}auss transform.
\newblock {\em SIAM J. Sci. Statist. Comput.}, 12:79--94, 1991.

\bibitem{greengard1998nfgt}
L.~Greengard and X.~Sun.
\newblock A new version of the fast {G}auss transform.
\newblock {\em Documenta Mathematica, III pp}, pages 575--584, 1998.

\bibitem{guenther1988}
R.~B. Guenther and J.~W. Lee.
\newblock {\em Partial differential equations of mathematical physics and
  integral equations}.
\newblock Prentice Hall, Inglewood Cliffs, New Jersey, 1988.

\bibitem{gustafsson}
Bertil Gustafsson, Heinz-Otto Kreiss, and Joseph Oliger.
\newblock {\em Time-dependent problems and difference methods}.
\newblock Pure and Applied Mathematics. John Wiley \& Sons, Hoboken, NJ, 2nd
  edition, 2013.

\bibitem{gustafsson1972mcom}
Bertil Gustafsson, Heinz-Otto Kreiss, and Arne Sundstr\"{o}m.
\newblock Stability theory of difference approximations for mixed initial
  boundary value problems. {II}.
\newblock {\em Math. Comp.}, 26:649--686, 1972.

\bibitem{Hackbusch2015}
Wolfgang Hackbusch.
\newblock {\em Hierarchical Matrices: Algorithms and Analysis}.
\newblock Springer, 2015.

\bibitem{hansen}
E.~Hansen and A.~Ostermann.
\newblock Exponential splitting for unbounded operators.
\newblock {\em Math. Comp.}, 78:1485--1496, 2009.

\bibitem{harbrecht2011ip}
Helmut Harbrecht and Johannes Tausch.
\newblock An efficient numerical method for a shape-identification problem
  arising from the heat equation.
\newblock {\em Inverse Problems}, 27(6):065013, 18, 2011.

\bibitem{harbrecht2014}
Helmut Harbrecht and Johannes Tausch.
\newblock On shape optimization with parabolic state equation.
\newblock In {\em Trends in {PDE} constrained optimization}, volume 165 of {\em
  Internat. Ser. Numer. Math.}, pages 213--229. Birkh\"{a}user/Springer, Cham,
  2014.

\bibitem{ho2012sisc}
K.~L. Ho and L.~Greengard.
\newblock A fast direct solver for structured linear systems by recursive
  skeletonization.
\newblock {\em SIAM J. Sci. Comput.}, 34(5):A2507--A2532, 2012.

\bibitem{ho2016cpam2}
Kenneth~L. Ho and Lexing Ying.
\newblock Hierarchical interpolative factorization for elliptic operators:
  differential equations.
\newblock {\em Comm. Pure Appl. Math.}, 69(8):1415--1451, 2016.

\bibitem{ho2016cpam1}
Kenneth~L. Ho and Lexing Ying.
\newblock Hierarchical interpolative factorization for elliptic operators:
  integral equations.
\newblock {\em Comm. Pure Appl. Math.}, 69(7):1314--1353, 2016.

\bibitem{huang2006jcp}
Jingfang Huang, Ming-Chih Lai, and Yang Xiang.
\newblock An integral equation method for epitaxial step-flow growth
  simulations.
\newblock {\em J. Comput. Phys.}, 216(2):724--743, 2006.

\bibitem{ibanez2002}
M.~T. Ibanez and H.~Power.
\newblock An efficient direct {BEM} numerical scheme for phase change problems
  using {F}ourier series.
\newblock {\em Computer Methods in Applied Mechanics and Engineering},
  191:2371--2402, 2002.

\bibitem{jahnke}
T.~Jahnke and C.~Lubich.
\newblock Error bounds for exponential operator splittings.
\newblock {\em BIT}, 40:735--744, 2000.

\bibitem{jiang2015acom}
S.~Jiang, L.~Greengard, and S.~Wang.
\newblock Efficient separated sum-of-exponentials approximations for the heat
  kernels and their applications.
\newblock {\em Adv. Comput. Math.}, 41(3):529--551, 2015.

\bibitem{langston2011camcs}
H.~Langston, L.~Greengard, and D.~Zorin.
\newblock A free-space adaptive {FMM}-based pde solver in three dimensions.
\newblock {\em Comm. Appl. Math. Comp. Sci.}, 6:79--122, 2011.

\bibitem{Lee2006}
Dongryeol Lee, Alexander Gray, and Andrew Moore.
\newblock Dual-tree fast gauss transforms.
\newblock {\em Advances in Neural Information Processing Systems}, 18:747--754,
  2006.

\bibitem{nufftlib}
J.~Y. Lee, L.~Greengard, and Z.~Gimbutas.
\newblock {NUFFT} {V}ersion 1.3.2 {S}oftware {R}elease.
\newblock \url{http://www.cims.nyu.edu/cmcl/nufft/nufft.html}, 2009.

\bibitem{GLee98}
June-Yub Lee and Leslie Greengard.
\newblock {A direct adaptive Poisson solver of arbitrary order accuracy}.
\newblock {\em Journal of computational physics}, 125:415--424, 1996.

\bibitem{nufft7}
J.Y. Lee and L.~Greengard.
\newblock The type 3 nonuniform {FFT} and its applications.
\newblock {\em J. Comput. Phys.}, 206:1--5, 2005.

\bibitem{li2009sisc}
J.~R. Li and L.~Greengard.
\newblock High order accurate methods for the evaluation of layer heat
  potentials.
\newblock {\em SIAM J. Sci. Comput.}, 31:3847--3860, 2009.

\bibitem{li2007jcp}
Jing-Rebecca Li and Leslie Greengard.
\newblock On the numerical solution of the heat equation. {I}. {F}ast solvers
  in free space.
\newblock {\em J. Comput. Phys.}, 226(2):1891--1901, 2007.

\bibitem{ying2017rms}
Yingzhou Li and Lexing Ying.
\newblock Distributed-memory hierarchical interpolative factorization.
\newblock {\em Res. Math. Sci.}, 4:Paper No. 12, 23, 2017.

\bibitem{lin1993thesis}
P.~Lin.
\newblock {\em On the Numerical Solution of the Heat Equation in Unbounded
  Domains}.
\newblock PhD thesis, Courant Institute of Mathematical Sciences, New York
  University, New York, 1993.

\bibitem{biros2015cicp}
D.~Malhotra and G.~Biros.
\newblock {PVFMM}: a parallel kernel independent {FMM} for particle and volume
  potentials.
\newblock {\em Commun. Comput. Phys.}, 18(3):808--830, 2015.

\bibitem{malhotra2016toms}
D.~Malhotra and G.~Biros.
\newblock Algorithm 967: A distributed-memory fast multipole method for volume
  potentials.
\newblock {\em ACM Trans. Math. Softw.}, 43:1--17, 2016.

\bibitem{mclachlan}
Robert~I. McLachlan and G.~Reinout~W. Quispel.
\newblock Splitting methods.
\newblock {\em Acta Numerica}, 11:341--434, 2002.

\bibitem{messner2014jcp}
Michael Messner, Martin Schanz, and Johannes Tausch.
\newblock A fast {G}alerkin method for parabolic space-time boundary integral
  equations.
\newblock {\em J. Comput. Phys.}, 258:15--30, 2014.

\bibitem{messner2015sisc}
Michael Messner, Martin Schanz, and Johannes Tausch.
\newblock An efficient {G}alerkin boundary element method for the transient
  heat equation.
\newblock {\em SIAM J. Sci. Comput.}, 37(3):A1554--A1576, 2015.

\bibitem{ying2017siammms}
Victor Minden, Kenneth~L. Ho, Anil Damle, and Lexing Ying.
\newblock A recursive skeletonization factorization based on strong
  admissibility.
\newblock {\em Multiscale Model. Simul.}, 15(2):768--796, 2017.

\bibitem{murua}
A.~Murua and J.~Sanz-Serna.
\newblock Order conditions for numerical integrators obtained by composing
  simpler integrators.
\newblock {\em Phil. Trans. R. Soc. Lond. A}, 357:1079--1100, 1999.

\bibitem{osher1969mcom}
Stanley Osher.
\newblock Stability of difference approximations of dissipative type for mixed
  initial-boundary value problems. {I}.
\newblock {\em Math. Comp.}, 23:335--340, 1969.

\bibitem{pogorzelski}
W.~Pogorzelski.
\newblock {\em Integral equations and their applications}.
\newblock Pergamon Press, Oxford, 1966.

\bibitem{nufft8}
D.~Potts, G.~Steidl, and M.~Tasche.
\newblock Fast fourier transforms for nonequispaced data: A tutorial.
\newblock {\em Modern Sampling Theory Mathematics and Applications}, pages
  249--274, 2001.

\bibitem{sampath2010pfgt}
R.~S. Sampath, H.~Sundar, and S.~Veerapaneni.
\newblock Parallel fast {G}auss transform.
\newblock {\em in SC Proceedings of the ACM/IEEE International Conference for
  High Performance Computing, Networking, Storage and Analysis, New Orleans,
  LA}, 10:1--10, 2010.

\bibitem{sethianstrain}
James~A. Sethian and John Strain.
\newblock Crystal growth and dendritic solidification.
\newblock {\em J. Comput. Phys.}, 98:231--253, 1992.

\bibitem{spivak2010sisc}
Marina Spivak, Shravan~K. Veerapaneni, and Leslie Greengard.
\newblock The fast generalized {G}auss transform.
\newblock {\em SIAM J. Sci. Comput.}, 32(5):3092--3107, 2010.

\bibitem{strain1991vfgt}
J.~Strain.
\newblock The fast {G}auss transform with variable scales.
\newblock {\em SIAM J. Sci. Stat. Comput.}, 12:1131--1139, 1991.

\bibitem{strain1994sisc}
J.~Strain.
\newblock Fast adaptive methods for the free-space heat equation.
\newblock {\em SIAM J. Sci. Comput.}, 15:185--206, 1994.

\bibitem{strang}
Gilbert Strang.
\newblock On the construction and comparison of difference schemes.
\newblock {\em SIAM J. Numer. Anal.}, 5:506--517, 1968.

\bibitem{strikwerda}
John~C. Strikwerda.
\newblock {\em Finite difference schemes and partial differential equations}.
\newblock SIAM, Philadelphia, {PA}, 2004.

\bibitem{tausch2009sisc}
J.~Tausch and A.~Weckiewicz.
\newblock Multidimensional fast {G}auss transforms by {C}hebyshev expansions.
\newblock {\em SIAM J. Sci. Comput.}, 31:3547--3565, 2009.

\bibitem{tausch2007jcp}
Johannes Tausch.
\newblock A fast method for solving the heat equation by layer potentials.
\newblock {\em J. Comput. Phys.}, 224(2):956--969, 2007.

\bibitem{tausch2009anm}
Johannes Tausch.
\newblock Nystr\"{o}m discretization of parabolic boundary integral equations.
\newblock {\em Appl. Numer. Math.}, 59(11):2843--2856, 2009.

\bibitem{tausch2012}
Johannes Tausch.
\newblock Fast {N}ystr\"{o}m methods for parabolic boundary integral equations.
\newblock In {\em Fast boundary element methods in engineering and industrial
  applications}, volume~63 of {\em Lect. Notes Appl. Comput. Mech.}, pages
  185--219. Springer, Heidelberg, 2012.

\bibitem{trefethen}
L.~Trefethen.
\newblock Numerical computation of the {S}chwarz-{C}hristoffel transformation.
\newblock {\em SIAM J. Sci. Stat. Comput.}, 1:82--102, 1980.

\bibitem{trefethen1983jcp}
Lloyd~N. Trefethen.
\newblock Group velocity interpretation of the stability theory of
  {G}ustafsson, {K}reiss, and {S}undstr\"{o}m.
\newblock {\em J. Comput. Phys.}, 49(2):199--217, 1983.

\bibitem{tyson}
R.~Tyson, L.~G. Stern, and R.~J. LeVeque.
\newblock Fractional step methods applied to a chemotaxis model.
\newblock {\em J. Math. Biol.}, 41:455--475, 2000.

\bibitem{vazquez2017}
Juan~Luis V{\'a}zquez.
\newblock {\em The Mathematical Theories of Diffusion: Nonlinear and Fractional
  Diffusion}, pages 205--278.
\newblock Springer International Publishing, Cham, 2017.

\bibitem{veerapaneni2007sisc}
S.~K. Veerapaneni and G.~Biros.
\newblock A high-order solver for the heat equation in 1d domains with moving
  boundaries.
\newblock {\em SIAM J. Sci. Comput.}, 29:2581--2606, 2007.

\bibitem{veerapaneni2008jcp}
S.~K. Veerapaneni and G.~Biros.
\newblock The {C}hebyshev fast {G}auss and nonuniform fast {F}ourier transforms
  and their application to the evaluation of distributed heat potentials.
\newblock {\em J. Comput. Phys.}, 227:7768--7790, 2008.

\bibitem{wan2006jcp}
X.~Wan and G.~Karniadakis.
\newblock A sharp error estimate for the fast {G}auss transform.
\newblock {\em Journal of Computational Physics}, 219:7--12, 2006.

\bibitem{wang2017thesis}
J.~Wang.
\newblock {\em Integral equation methods for the heat equation in moving
  geometry}.
\newblock PhD thesis, Courant Institute of Mathematical Sciences, New York
  University, New York, September 2017.

\bibitem{wang2019acom}
J.~Wang and L.~Greengard.
\newblock Hybrid asymptotic/numerical methods for the evaluation of layer heat
  potentials in two dimensions.
\newblock {\em Adv. Comput. Math.}, 45(2):847--867, 2019.

\bibitem{wang2018sisc}
Jun Wang and Leslie Greengard.
\newblock An adaptive fast {G}auss transform in two dimensions.
\newblock {\em SIAM J. Sci. Comput.}, 40(3):A1274--A1300, 2018.

\bibitem{wang2016njit}
S.~Wang.
\newblock {\em Efficient High-Order Integral Equation Methods for the Heat
  Equation}.
\newblock PhD thesis, Department of Mathematical Sciences, New Jersey Institute
  of Technology, Newark, New Jersey, August 2016.

\bibitem{wang2019jsc}
Shaobo Wang, Shidong Jiang, and Jing Wang.
\newblock Fast high-order integral equation methods for solving boundary value
  problems of two dimensional heat equation in complex geometry.
\newblock {\em Journal of Scientific Computing}, 79(2):787--808, 2019.

\bibitem{ying2004jcp}
L.~Ying, G.~Biros, and D.~Zorin.
\newblock A kernel-independent adaptive fast multipole algorithm in two and
  three dimensions.
\newblock {\em J. Comput. Phys.}, 196:591--626, 2004.

\bibitem{zhang2011jcp}
Bo~Zhang, Jingfang Huang, Nikos~P. Pitsianis, and Xiaobai Sun.
\newblock A {F}ourier-series-based kernel-independent fast multipole method.
\newblock {\em J. Comput. Phys.}, 230(15):5807--5821, 2011.

\end{thebibliography}

\label{lastpage}
\end{document}